\documentclass[12pt]{amsart}
\usepackage{amsfonts}
\usepackage{amssymb, amscd, amsthm}
\usepackage{latexsym}
\usepackage{multirow}
\usepackage{amsmath}
\usepackage[all]{xy}
\usepackage{mathrsfs}
\usepackage{amsbsy}
\usepackage{amsfonts}
\usepackage{mathrsfs}
\usepackage{verbatim}
\usepackage{version}
\usepackage{color}

\newtheorem{theorem}{Theorem}[section]
\newtheorem{lemma}[theorem]{Lemma}
\newtheorem{thm}[theorem]{Theorem}
\newtheorem{prop}[theorem]{Proposition}
\newtheorem{rem}[theorem]{Remark}
\newtheorem{coro}[theorem]{Corollary}
\newtheorem{defn}[theorem]{Definition}
\newtheorem{example}[theorem]{Example}

\newtheorem{con/que}[theorem]{Conjecture/Question}

\newcommand{\ra}{\rightarrow}
\newcommand{\mo}{\mathcal{O}}
\newcommand{\mf}{\mathcal{F}}
\newcommand{\mg}{\mathcal{G}}
\newcommand{\ma}{\mathcal{A}}
\newcommand{\mb}{\mathcal{B}}
\newcommand{\me}{\mathcal{E}}

\newcommand{\mext}{\mathbb{E}\mathtt{x}\mathtt{t}}
\newcommand{\mhom}{\mathbb{H}\mathsf{o}\mathsf{m}}
\newcommand{\ms}{\mathcal{S}}
\newcommand{\mt}{\mathcal{T}}
\newcommand{\mh}{\mathcal{H}}
\newcommand{\mm}{\mathcal{M}}
\newcommand{\wmn}{\widetilde{\mathcal{N}}}
\newcommand{\mn}{\mathcal{N}}
\newcommand{\mw}{\mathcal{W}}
\newcommand{\wmv}{\widetilde{\mathcal{V}}}
\newcommand{\mv}{\mathcal{V}}
\newcommand{\um}{\mathcal{U}}

\newcommand{\ts}{\mathbf{S}}

\newcommand{\mr}{\mathcal{R}}

\newcommand{\E}{\mathscr{E}}

\newcommand{\Hom}{\operatorname{Hom}}
\newcommand{\Ext}{\operatorname{Ext}}

\def\<{\langle}
\def\>{\rangle}

\newcommand{\ls}{|L|}
\newcommand{\p}{\mathbb{P}}
\newcommand{\bz}{\mathbb{Z}}
\newcommand{\bq}{\mathbb{Q}}
\newcommand{\bc}{\mathbb{C}}
\newcommand{\bl}{\mathbb{L}}
\newcommand{\bu}{\mathbb{U}}
\newcommand{\bh}{\mathbb{H}}

\begin{document}
\fontsize{12pt}{14pt} \textwidth=14cm \textheight=21 cm
\numberwithin{equation}{section}
\title{Some Betti numbers of the moduli of 1-dimensional sheaves on $\mathbb{P}^2$.}
\author{Yao Yuan}
\thanks{The author is supported by NSFC 21022107.  }

\begin{abstract}Let $M(d,\chi)$ with $(d,\chi)=1$ be the moduli space of semistable sheaves on $\mathbb{P}^2$ supported on curves of degree $d$ and with Euler characteristic $\chi$.    The cohomology ring $H^*(M(d,\chi),\mathbb{Z})$ of $M(d,\chi)$ is isomorphic to its Chow ring $A^*(M(d,\chi))$ by Markman's result.  W. Pi and J. Shen have described a minimal generating set of $A^*(M(d,\chi))$ consisting of $3d-7$ generators, which they also showed to have no relation in $A^{\geq d-2}(M(d,\chi))$.  We compute the two Betti numbers $b_{2(d-1)}$ and $b_{2d}$ of $M(d,\chi)$ and as a corollary we show that the generators given by Pi-Shen have no relations in $A^{\geq d-1}(M(d,\chi))$ but do have three linearly independent relations in $A^d(M(d,\chi))$.

~~~

\end{abstract}

\maketitle
\tableofcontents
\section{Introduction.}

\subsection{Motiviations.}We work over the complex numbers $\bc$.

The moduli space $M(L,\chi)^{ss}$ parametrizes 1-dimensional (Gieseker) semistable sheaves of rank 0, determinant $L$ and Euler characteristic $\chi$ over a projective surface $S$.  All (S-equivalence classes of) sheaves in $M(L,\chi)^{ss}$ are supported on curves in the linear system $\ls$.  Let $M(L,\chi)\subset M(L,\chi)^{ss}$ parametrize stable sheaves in $M(L,\chi)^{ss}$.

$M(L,\chi)^{ss}$ is closely related to the moduli space of Higgs bundles if $S$ is a ruled surface, and to PT-theory on a local Calabi-Yau 3-fold if $S$ is Fano.   Therefore it is a very interesting problem to compute topological invariants of $M(L,\chi)^{ss}$ such as the Euler number and the Betti numbers.  

Although there are general formulas for the (virtual) Poincar\'e polynomials of moduli spaces of Higgs bundles (see \cite{HRV}, \cite{Mel}, \cite{MoSc} and \cite{Sch}), we don't have any general formula for the Poincar\'e polynomial or the Euler number of $M(L,\chi)^{ss}$.   Only for special cases, all Betti numbers of $M(L,\chi)^{ss}$ are known, such as \cite{Yuan3}.  Results in \cite{Bou1}, \cite{Bou2}, \cite{MS} and \cite{Yuan9} help us understand better the enumerative geometry of $M(L,\chi)^{ss}$, such as $\chi$-independence of the intersection cohomology groups for $S$ a del Pezzo surface.  But still we don't know how to compute the dimensions of those intersection cohomology groups.  

Let $S=\p^2$ and $L=dH$ with $H$ the hyperplane class.  If $(d,\chi)=1$, then $M(d,\chi)^{ss}=M(L,\chi)$ and by \cite[Theorem 1]{mark} the cohomology ring $H^*(M(d,\chi),\bz)$ is torsion-free and isomorphic to the Chow ring $A^*(M(d,\chi))$.  In \cite{PS}, W. Pi and J. Shen studied the Chow ring $A^*(M(d,\chi))$ and found a minimal generating set consisting of $3d-7$ generators.  They also showed that those generators don't have relations in $A^{\leq d-2}(M(d,\chi))$.  

In this paper we study the freeness of those generators in degrees $d-1$ and $d$ by computing the Betti numbers $b_{2(d-1)}, b_{2d}$ of $M(d,\chi)$.  Our result shows that the relation of the least degree among those generators appears in $A^d(M(d,\chi))$.  Therefore it should not be easy to get all the Betti numbers of $M(d,\chi)$ by studying $A^{*}(M(d,\chi))$.  By \cite[Theorem 0.1]{MS}, we have isomorphisms of graded vector spaces $H^*(M(d,\chi),\bq)\cong H^*(M(d,\chi'),\bq)$ for any $\chi,\chi'$ coprime to $d$.  This opens the question whether the ring structure of  $H^*(M(d,\chi),\bz)$ is also $\chi$-independent as long as $\chi$ is coprime to $d$.  Hence to study the Chow ring is itself an interesting problem.

\subsection{Results \& Applications.}
Let  $d\geq 5$, and $\chi=-d-1$. 
Let $(\p^2)^{[m]}$ be the Hilbert scheme of $m$ points on $\p^2$.  Our main result is the following theorem.
\begin{thm}[Theorem \ref{main}]\label{intro1}For $d\geq 5$ and $\chi$ coprime to $d$, we have
\[b_{2k}(M(d,\chi))=\left\{\begin{array}{ll}b_{2k}((\p^2)^{[n]}),&k\leq d-2\\ b_{2k}((\p^2)^{[n]})-3,&k=d-1\\b_{2k}((\p^2)^{[n]})-12,&k= d\end{array}\right.,\]
where $n=\frac{d(d-1)}2+1.$
\end{thm}

Notice that we have already obtained $b_{2k}(M(d,\chi)),k\leq d-2$ in \cite{Yuan9} (see \S\ref{generator}).  

By \cite[Theorem 0.1]{MS}, the Betti numbers of $M(d,\chi)$ are $\chi$-independent  for any $\chi$ coprime to $d$.  Hence Theorem \ref{intro1} holds for any $\chi$ coprime to $d$ if it holds for $\chi=-d-1$.  We get the following theorem as a corollary to Theorem \ref{intro1} which proves Conjecture 3.3 in \cite{PS}.

\begin{thm}[Corollary \ref{mainco}]\label{intro2}For $d\geq 5$ and $\chi$ coprime to $d$, the $3d-7$ generators 
$$c_0(2),c_2(0),c_k(0),c_{k-1}(1),c_{k-2}(2), k\in {3,\cdots,d-1}$$ of $A^*(M(d,\chi))\cong H^*(M(d,\chi),\mathbb{Z})$ given in \cite{PS} have no relation in $A^i(M(d,\chi)),i\leq d-1$ and have 3 linearly independent relations in $A^d(M(d,\chi))$. 
\end{thm}

Notice that to get Theorem \ref{intro2}, it is sufficient and necessary to know $b_{2(d-1)}(M(d,\chi))$ and $b_{2d}(M(d,\chi))$ (see \S \ref{generator}).

Theorem \ref{intro1} is proved by studying the class of $[M(d,\chi)]$ in the Grothendieck ring $\widehat{K}(Var_{\bc})$ of stacks over $\bc$ (see \S\ref{Gring}). 
Let $\widehat{K}_m$ be the subgroup (with the group structure given by the addition) of $\widehat{K}(Var_{\bc})$ consisting of stacks with dimension $\leq m$. For $[\ma],[\mb]\in\widehat{K}(Var_{\bc})$, we write $[\ma]\equiv[\mb]~~(\widehat{K}_m)$ if $[\ma]-[\mb]\in \widehat{K}_m$.

We have the following theorem.
\begin{thm}[Theorem \ref{main}]\label{intro3}Let $d\geq 5$ and $\chi=-d-1$.  In $\widehat{K}(Var_{\bc})$ we have
\[\bl^{d-1}[(\p^2)^{[\frac{d(d-1)}2+1]}]-[M(d,\chi)]\equiv3[\p^{2d-4}][\p^2][(\p^2)^{[\frac{(d-1)(d-2)}2+1]}]~~(\widehat{K}_{d^2-d}),\]
 where $\bl:=[\mathbb{A}^1]$ is the class of an affine line in $\widehat{K}(Var_{\bc})$.
\end{thm}

\subsection{Notations \& Conventions.}
\begin{itemize}
\item
Let $S$ be a smooth projective surface over $\bc$ with $H$ an ample line bundle and $K_S$ the canonical line bundle. 

\item We use the same letters for line bundles and the corresponding divisor classes.  For instance, $L_1\otimes L_2$ is the tensor of two line bundles, and $L_1.L_2$ is the intersection number of their divisor classes.  Define $L_1^2=L_1.L_1$.

\item Let $M(L,\chi)^{ss}$ be the moduli space parametrizing 1-dimensional (Gieseker) semistable sheaves of rank 0, determinant $L$ and Euler characteristic $\chi$ over the projective surface $S$.  Let $M(L,\chi)\subset M(L,\chi)^{ss}$ parametrize stable sheaves in $M(L,\chi)^{ss}$.

\item We write $M(d,\chi)^{ss}$ ($M(d,\chi)$, resp.) instead of $M(dH,\chi)^{ss}$
($M(dH,\chi)$, resp.) if $S=\p^2$ and $H$ is the hyperplane class.

\item For a sheaf $\mf$, $\chi(\mf)$ is its Euler characteristic and $h^i(\mf):=\dim H^i(\mf)$.


\item For a nontrivial effective line bundle $L$ on $S$, we have some notations as follows. \begin{itemize}
\item
Let $\ls$ be the linear system of the corresponding divisor class.  
\item Let $g_L:=\frac{L.(L+K_S)}2+1$ be the arithmetic genus of curves in $\ls$.  
\item Denote by $\ls^{int}$ the open subset of $\ls$ consisting of all integral curves.  
\item Let $\rho_L:=\dim\ls-\dim(\ls\setminus\ls^{int})$, i.e. $\rho_L$ is the codimension of the subset inside $\ls$ consisting of all non-integral curves.  \end{itemize}

\end{itemize}

\subsection{Plan of the paper.} Section 2 of the paper provides preliminaries. In \S 2.1 we give some basic facts on 1-dimensional sheaves over surfaces and their moduli.  In \S 2.2 we review the result in \cite{PS} for the generators of the cohomology ring of $M(d,\chi)$.  In \S 2.3 we give a brief introduction to the motivic measures of algebraic stacks.  In Section 3 we define some stacks (\S 3.1) and prove some technic results in dimension estimate (\S 3.2), and finally study the sheaves supported on curves with two integral components (\S 3.3).  In Section 4 we prove our main theorem, where there are also some technic lemmas and propositions.  

\subsection{Acknowledgements.} I would like to thank Weite Pi and Junliang Shen for their paper \cite{PS} which motivates me for this work.  
I thank the referees for all the valuable comments.

\section{Preliminaries.}

\subsection{Moduli spaces of 1-dimensional sheaves on a projective surface.} Let $S$ be a smooth projective surface over $\bc$ with $H$ an ample line bundle and $K_S$ the canonical line bundle.  Let $L$ be an effective non-trivial line bundle on $S$.  

Let $\mf$ be a sheaf of rank 0 and determinant $L$, then $\mf$ is a \textbf{1-dimensional sheaf} on $S$.  
If $\mf$ is pure, i.e. it does not contain non-trivial 0-dimensional subsheaves, then $\mf$ has a locally free resolution of length one (see e.g. \cite[Lemma 4.5.13]{Yuan1}) as follows
\begin{equation}\label{lfr}0\ra\ma\xrightarrow{\alpha}\mb\ra\mf\ra0,\end{equation}
with $\ma,\mb$ locally free sheaves.  

Locally the map $\alpha$ in (\ref{lfr}) is given by a square matrix, hence we can define the determinant $\det(\alpha)$ which is a section of $\det(\ma)^{-1}\otimes\det(\mb)\cong \det(\mf)\cong L$.  It is easy to see that $\det(\alpha)$ vanishes at a point $s\in S$ iff the stalk $\mf_s\neq 0$.  The section $\det(\alpha)$ of $L$ defines a unique curve in the linear system $\ls$, which is called the \textbf{schematic support} or \textbf{Fitting support} of $\mf$, and is denoted by supp$(\mf)$.   

The 1-dimensional scheme supp$(\mf)$ sometimes is not a variety but a scheme.  For instance, if $\mf$ is a rank $r\geq2$ sheaf on an integral curve $C_0\subset S$, then supp$(\mf)=rC_0$ is a non-reduced curve.  We will call supp$(\mf)$ the support of $\mf$ for short if there is no confusion.  

\begin{example}\label{supp}For any curve $C$ in $\ls$, one can a find pure 1-dimensional sheaf $\mf$ with \emph{supp}$(\mf)=C$.  Let $C=n_1 C_1 \cup \cdots\cup n_s C_s$ with $C_1,\cdots,C_s$ pairwise distinct integral curves.  Let $\mf_i$ be a rank $n_i$ bundle over $C_i$, $i=1,\cdots,s$, then $\mf=\bigoplus_{i=1}^s\mf_i$ has $C$ as its support. 
\end{example}
\begin{rem}If $\mf$ is not purely 1-dimensional, denote by $\mt$ its maximal 0-dimensional subsheaf and let $C=$ \emph{supp}$(\mf/\mt)$.  However $\mf$ is not necessarily an $\mo_C$-module since $\mt_x$ may not be zero for some point $x$ outside $C$.  According to the convention in \cite{Yuan9}, we say $\mf$ has $C$ as its support if $\mf$ is a $\mo_C$-module and \emph{supp}$(\mf/\mt)=C$.
\end{rem}

For 1-dimensional sheaves, the Gieseker semistability coincides with the slope semistability.   For a 1-dimensional sheaf $\mf$, its slope is $\mu(\mf):=\frac{\chi(\mf)}{\det(\mf).H}$.  $\mf$ is (semi)stable if for every $0\neq \mf_1\subsetneq \mf$, 
$\mu(\mf_1)<(\leq)\mu(\mf)$. It is easy to see that semistability implies purity.

Let $\chi$ be an integer.  Let $M(L,\chi)^{ss}$ be the moduli space of semistable 1-dimensional sheaves with determinant $L$ and Euler characteristic $\chi$.  Then we have the Hilbert-Chow morphism
\begin{equation}\label{hcm}\pi^{ss}:M(L,\chi)^{ss}\ra\ls,~\mf\mapsto \text{supp}\mf.\end{equation}
The map $\pi^{ss}$ in (\ref{hcm}) is not only a set-theoretic map but also a morphism of algebraic schemes (see e.g. \cite[Proposition 3.0.2]{Yuan1}).  The fiber of $\pi^{ss}$ over an integral curve $C$ is isomorphic to the (compactified) Jacobian of $C$ while the fibers over non-integral curves can have more than one irreducible component.  Let $M(L,\chi)\subset M(L,\chi)^{ss}$ be the open subscheme parametrizing stable sheaves.  Denote by \begin{equation}\label{hcms}\pi:M(L,\chi)\ra\ls\end{equation} the Hilbert-Chow morphism restricted to $M(L,\chi)$.  We have the following result.
\begin{prop}[Corollary 1.3 in \cite{Yuan9}]\label{dimest}If $S$ is Fano or with $K_S$ trivial, and if moreover $\ls$ contains an integral curve, then all the fibers of the Hilbert-Chow morphism $\pi$ in (\ref{hcms}) have dimension $g_L$, where $g_{L}$ is the arithmetic genus of any curve in $\ls$. 
\end{prop}

 Let $N(L,\chi)\subset M(L,\chi)^{ss}$ be the subscheme parametrizing sheaves with integral supports.  Then $N(L,\chi)\subset M(L,\chi)$ and also $N(L,\chi)$ is open in $M(L,\chi)^{ss}$, which is because the set $|L|^{int}\subset |L|$ of integral curves is open in $|L|$ and $N(L,\chi)=(\pi^{ss})^{-1}(|L|^{int})$.
\begin{rem}\label{dimp}If $S$ is Fano, then $M(L,\chi)$ is either empty or smooth of the expected dimension $L^2+1$.  If moreover $\ls$ contains an integral curve, i.e. $\ls^{int}\neq\emptyset$, then $N(L,\chi)$ is of dimension $\dim\ls+g_L$, and so is $M(L,\chi)$ by Proposition \ref{dimest}.  By a direct computation, we see in this case $h^0(L)=\chi(L)$ and $H^1(L)=0$.  \end{rem}

Universal sheaves don't always exist even over $S\times M(L,\chi)$.  Let $S=\p^2$ with $H$ the hyperplane class, and let $L=dH$, then there is a universal sheaf over $S\times M(L,\chi)$ iff $(d,\chi)=1$ which is equivalent to $M(L,\chi)=M(L,\chi)^{ss}$ (See \cite[Theorem 3.19]{LP1}).     

\subsection{Generators for the cohomology ring of $M(d,\chi)$.}\label{generator}
Let $S=\p^2$ with $H$ the hyperplane class, and let $L=dH$ and $(d,\chi)=1$.  Then $M(d,\chi):=M(L,\chi)$ is projective and smooth of dimension $d^2+1$.  Denote by $\E$ a universal sheaf over $S\times M(d,\chi)$.  For every Chern class $c_i(\E)$ of $\E$, we have the K\"unneth decomposition
\[c_i(\E)=H^0\otimes e_i(2)+H\otimes e_i(1)+H^2\otimes e_i(0)\]
with $e_i(j)\in A^{i+j-2}(M(d,\chi))$.  Notice that $H^*(\p^2,\bz)\cong \bz[H]/(H^3)$. 

By \cite[Theorem 1 and Theorem 2]{mark},  we have the cohomology ring $H^*(M(d,\chi),\bz)$ is torsion-free and isomorphic to the Chow ring $A^*(M(d,\chi))$.  Hence the odd Betti numbers of $M(d,\chi)$ are all zero.  Moreover $H^*(M(d,\chi),\bz)$ is generated by the K\"unneth factors $e_i(j)$ of all the Chern classes of $\E$ (see \cite[Proposition]{Bea} or \cite[Proposition 12]{mark}).

W. Pi and J. Shen studied the relation between those generators $e_i(j)$ in \cite{PS} and they proved that a minimal set of generators can be chosen as 
$$\Sigma:=\{c_0(2),c_2(0),c_k(0),c_{k-1}(1),c_{k-2}(2), k\in {3,\cdots,d-1}\}$$ (See \cite[Theorem 0.2]{PS}).  Here the tautological class $c_k(j)\in A^{k+j-1}(M(d,\chi))$ is defined as the K\"unneth factor of the degree $k+1$ component ch$_{k+1}^{\alpha}(\E)$ of the twisted Chern character ch$^{\alpha}(\E):=\text{ch}(\E)\cdot \text{exp}(\alpha)$.  The class $\alpha\in A^1(\p^2\times M(d,\chi))$ is uniquely determined by asking $c_1(0)=0\in A^0(M(d,\chi))$ and $c_1(1)=0\in A^1(M(d,\chi))$.  Notice that we always have $c_0(1)=d\in A^0(M(d,\chi))$ for any $\alpha$.  

We write down the generating series 
\begin{equation}\label{gs1}G(z):=\prod_{c_k(j)\in \Sigma}\sum_{l\geq0}\left(c_k(j)z^{2(k+j-1)}\right)^l=\prod_{c_k(j)\in\Sigma}\left(1-c_k(j)z^{2(k+j-1)}\right)^{-1}.\end{equation}
Then the coefficients of monomials with degree $2i$ in $G(z)$ span $A^i(M(d,\chi))\cong H^{2i}(M(d,\chi),\bz)$ for $i\leq d^2+1$.  Denote by $a_{2i}$ the number of monomials with degree $2i$ in $G(z)$.  Then $a_{2i}$ is the coefficient of $z^{2i}$ in the following series
\begin{equation}\label{ags1}F(z):=\left(\sum_{l\geq0}\left(z^2\right)^l\right)^2\cdot \prod_{k=2}^{d-2}\left(\sum_{l\geq 0}\left(z^{2k}\right)^l\right)^3.
\end{equation}

We already have known that $H^{2i}(M(d,\chi),\bz)$ is torsion-free.  Therefore, in order to see the freeness of these generators in $A^i(M(d,\chi))$, it is enough to compare $a_{2i}$ with the $2i$-th Betti number $b_{2i}(M(d,\chi))$.

For $i\leq d-2$, $b_{2i}(M(d,\chi))$ are given by the following theorem.
     
\begin{thm}[Theorem 1.7 in \cite{Yuan9}]\label{genY}Let $M_{\p^2}(d,\chi)$ be the coarse moduli space of 1-dimensional semistable sheaves with schematic supports in $|dH|$ and Euler characteristic $\chi$.  Then for $0\leq k\leq 2d-3$ we have
\[\left\{\begin{array}{ll}b_k^v(M_{\p^2}(d,\chi))=0,&\text{ for }k\text{ odd }\\
b_k^v(M_{\p^2}(d,\chi))=b_{k}^v\left((\p^2)^{[\frac{d(d-3)}2-\chi_0]}\right),&\text{ for }k\text{ even }\\
h^{p,k-p}_v(M_{\p^2}(d,\chi))=0,&\text{ for }p\neq k-p\\
h^{p,k-p}_v(M_{\p^2}(d,\chi))=h^{p,i-p}_v\left((\p^2)^{[\frac{d(d-3)}2-\chi_0]}\right),&\text{ for }k=2p
\end{array}\right.\]
where $b_k^v$ ($h^{p,q}_v$, resp.) denotes the $k$-th virtual Betti number ($(p,q)$-th virtual Hodge number, resp.) (see Definition \ref{dvhb}),  $(\p^2)^{[n]}$ is the Hilbert scheme of $n$-points on $\p^2$ and finally $\chi_0\equiv \chi~(d)$ with $-2d-1\leq \chi_0\leq -d+1$.
\end{thm}
\begin{rem}\label{chi0}Notice that in Theorem \ref{genY}, the choice of $\chi_0$ is not unique for $\chi=0,1,-1$.  Therefore without loss of generality, we can assume $-2d\leq\chi_0\leq -d-1$.\end{rem}

By the well-known formula due to G\"ottsche (see \cite{Goe}), we have 
\begin{equation}\label{genG}\sum_{n\geq 0}\sum_{i\geq 0}b_i\left(\left(\p^2\right)^{[n]}\right)z^it^n=\prod_{k\geq 1}\left(1-z^{2k-2}t^k\right)^{-1}\left(1-z^{2k}t^k\right)^{-1}\left(1-z^{2k+2}t^k\right)^{-1}.
\end{equation}
We can rewrite the right hand side of (\ref{genG}) as follows.
\begin{eqnarray}\label{RHS}&&R(z,t):=RHS=\prod_{k\geq 1}\left(\sum_{l\geq0} \left(z^{2k-2}t^k\right)^l \right)\left(\sum_{l\geq0} \left(z^{2k}t^k\right)^l \right)\left(\sum_{l\geq0} \left(z^{2k+2}t^k\right)^l \right)\nonumber\\
&&=\left(1+t+t^2+\cdots\right)\cdot\prod_{k\geq 1}\left(\sum_{l\geq0} \left(z^{2k}t^k\right)^l \right)\left(\sum_{l\geq0} \left(z^{2k+2}t^k\right)^l \right)\left(\sum_{l\geq0} \left(z^{2k}t^{k+1}\right)^l \right).\nonumber\\\end{eqnarray}
The Betti number $b_k\left(\left(\p^2\right)^{[n]}\right)$ is the coefficient of $z^kt^n$ in $R(z,t)$.  By (\ref{RHS}) we have $b_k\left(\left(\p^2\right)^{[n]}\right)=0$ if $k$ odd and $b_{k}\left(\left(\p^2\right)^{[n]}\right)$ does not depend on $n$ if $k\leq n$.  In fact the coefficient of $z^{2i}t^n$ in $R(z,t)$ equals to the sum of all coefficients of $z^{2i}t^m$ with $m\leq n$ in the following series
$$\widetilde{R}(z,t)=\prod_{k\geq 1}\left(\sum_{l\geq0} \left(z^{2k}t^k\right)^l \right)\left(\sum_{l\geq0} \left(z^{2k+2}t^k\right)^l \right)\left(\sum_{l\geq0} \left(z^{2k}t^{k+1}\right)^l \right).$$
It is easy to see that the expansion of $\widetilde{R}(z,t)$ only contains terms $z^{2j}t^u$ with $u\leq 2j$.
Therefore for $2i\leq 2s\leq n$, the sum of all coefficients of $z^{2i}t^m$ with $m\leq n$ in $\widetilde{R}(z,t)$ is equal to the coefficient of $z^{2i}$ in the following series
$$\widetilde{R}(z)=\prod_{k\geq 1}\left(\sum_{l\geq0} \left(z^{2k}\right)^l \right)\left(\sum_{l\geq0} \left(z^{2k+2}\right)^l \right)\left(\sum_{l\geq0} \left(z^{2k}\right)^l \right).$$
Again by $2i\leq 2s$, we may replace $\widetilde{R}(z)$ by the $R_s(z)$ defined as follows
\begin{eqnarray}\label{gs2}&&R_s(z):=\left(\sum_{l\geq0} \left(z^{2s}\right)^l \right)^2\cdot\prod_{k= 1}^{s-1}\left(\sum_{l\geq0} \left(z^{2k}\right)^l \right)\left(\sum_{l\geq0} \left(z^{2k+2}\right)^l \right)\left(\sum_{l\geq0} (z^{2k})^l \right)\nonumber\\ &&=\left(\sum_{l\geq0} \left(z^{2}\right)^l \right)^2\cdot \prod_{k=2}^{s}\left(\sum_{l\geq0} \left(z^{2k}\right)^l \right)^3.\end{eqnarray}
For $2i\leq2s\leq n$,  $b_{2i}\left(\left(\p^2\right)^{[n]}\right)$ is the coefficient of $z^{2i}$ in $R_s(z)$ in (\ref{gs2}).

Let $n=\frac{d(d-3)}2-\chi_0$ with $-2d\leq\chi_0\leq -d-1$.  Then $n\geq 2d$ for $d\geq5$. 

Let $d\geq 5$.  Recall that the number $a_{2i}$ of monomials with degree $2i$ in $G(z)$ equals to the coefficient of $z^{2i}$ in $F(z)$ in (\ref{ags1}).  By comparing (\ref{ags1}) and (\ref{gs2}) we have that $F(z)=R_s(z)$ if $s=d-2$ and hence
$a_{2i}=b_{2i}\left(\left(\p^2\right)^{[n]}\right)=b_{2i}(M(d,\chi))$ for all $i\leq d-2$.  Hence generators in $\Sigma$ have no relation in $A^i(M(d,\chi))$ for $i<d-1$.  This implies the minimality of $\Sigma$.

Let $s=d$, then $R_s(z)=F(z)\cdot\displaystyle{\left(\sum_{l\geq 0}\left(z^{2(d-1)}\right)^l\right)^3\cdot\left(\sum_{l\geq 0}\left(z^{2d}\right)^l\right)^3}$.  Since $F(z)=1+2z^2+\cdots$, we have that 
\begin{equation}\label{twob}a_{2i}=\begin{cases}b_{2i}\left(\left(\p^2\right)^{[n]}\right)-3,~~i=d-1\\ b_{2i}\left(\left(\p^2\right)^{[n]}\right)-9,~~i=d\end{cases}.\end{equation} 

Therefore by Theorem \ref{intro1}, generators in $\Sigma$ have no relation in $A^i(M(d,\chi))$ for $i\leq d-1$ and they do have 3 linearly independent relations in $A^d(M(d,\chi))$ (see Corollary \ref{mainco}).  This proves Conjecture 3.3 in \cite{PS}.

Finally for $d\leq 4$,  we have $M(1,\chi)\cong \p^2$, $M(2,\chi)\cong \p^5$ and $M(3,1)$ is isomorphic to the universal curve $\mathcal{C}_3$ in $\p^2\times|3H|$ (see \cite[Theorem 1.2]{Yuan3}), which is a projective bundle over $\p^2$ with fiber isomorphic to $\p^8$.  The Chow ring of $M(3,1)$ is described explicitly in \cite[\S 1.3]{PS}.  All Betti numbers of $M(4,1)$ are obtained in \cite[Theorem 1.3]{Yuan3}.  But the Chow ring of $M(4,1)$ is still not clear. 

\subsection{The motivic measures of algebraic stacks.}\label{Gring}
Although our final result is about the moduli scheme $M(L,\chi)$, most of the time we work on moduli stacks and specifically we compute the motivic measures of those stacks.  In this subsection, we give a brief introduction to the motivic measures of algebraic stacks.  One can look at \cite{Bri}, \cite{GHS}, \cite{Joy} and \cite{To} for more details.   

Let $K(Var_k)$ be the \textbf{Grothendieck ring of varieties} over a field $k$, which by definition is the free abelian group on isomorphism classes $[Y]$ modulo relations 
$$[Y]=[Z]+[Y\setminus Z]$$
with $Z\subset Y$ a closed subvariety.  One also can view $K(Var_k)$ as the Grothendieck ring of $k$-schemes of finite type by defining $[Y]=[Y^{red}]$, where $Y^{red}$ is the same topological space with $Y$ endowed with the reduced scheme structure.

If $X\ra Y$ is a locally trivial fibration with fiber $F$, then $[X]=[Y]\times [F]$ in $K(Var_k)$.

Let $A$ be a commutative ring.  An \textbf{$A$-valued motivic measure} is a ring homomorphism $\mu_A:K(Var_k)\ra A$  (see e.g. \cite[\S1]{kap}).    
\begin{example}\label{mom}(1) Let $k=\bc$. The function $[X]\mapsto e(X)$ is a $\bz$-valued motivic measure, where $e(X)$ is the Euler number of $X$.

(2) Let $k=\bc$. 
Taking E-polynomials is a $\bz[x,y]$-valued motivic measure. (see Definition \ref{dmhp}) 

(3) Let $k=\mathbb{F}_q$ be a finite field.  The function $[X]\mapsto \#X(\mathbb{F}_{q^n})$ is a $\bz$-valued motivic measure, where $\#X(\mathbb{F}_{q^n})$ is the number of $\mathbb{F}_{q^n}$-points on $X$.
\end{example}

We also call the value $\mu_A(X)$ the \textbf{motivic measure} of $X$ (via $\mu_A$).

One can also define the motivic measure of an algebraic stack of finite type with affine stabilizers.  Let $\bl:=[\mathbb{A}^1]$ be the class of the affine line.  Let $GL_n$ be the general linear group of rank $n$.  Then $[GL_n]=\prod_{k=0}^{n-1}(\bl^n-\bl^k)$.  Define $$\widehat{K}(Var_k):=K(Var_k)[[GL_n]^{-1}:n\geq 1].$$  
By \cite[Lemma 3.8 and Lemma 3.9]{Bri}, we have $\widehat{K}(Var_k)\cong K(Var_k)[\bl^{-1},(\bl^i-1)^{-1}:i\geq 1]$ and $\widehat{K}(Var_k)$ is isomorphic to the Grothendieck ring of algebraic stacks of finite type with affine stabilizers. 

Now suppose $\mm$ is a quotient stack defined by an action of $GL_n$ on a scheme $X$ of finite type, i.e. $\mm\cong [X/GL_n]$ (here the square brackets $[~]$ is the standard notation for stacks).  Then $\mm$ is an algebraic stack of finite type with affine stabilizers (see e.g. \cite[Proposition 3.5]{Bri}) and  
\[[\mm]:=\frac{[X]}{[GL_n]}\in\widehat{K}(Var_k).\]

Let $\mu_A$ be the motivic measures in Example \ref{mom} (2) (3), then $\mu_A$ extend to $\widehat{K}(Var_k)$ after localizing $A$ at $\mu_A(\bl)$ and $\mu_A(\bl^i-1)$.  Since $e(\bl^i-1)=0$, the function $[X]\mapsto e(X)$ apparently seems not to be well-defined over $\widehat{K}(Var_k)$.

Let $k=\bc$ and let $X$ be any $k$-variety.  According to Deligne's mixed Hodge theory (\cite{Del1,Del2}), $H^j(X,\bq)$ has a mixed Hodge structure for each $j$ , i.e. there is an increasing weight filtration 
\[0=W_{-1}\subset W_0\subset\cdots \subset W_{2j}=H^j(X,\bq) \]
and a decreasing Hodge filtration
\[H^j(X,\bc)=F^0\supset F^1\supset\cdots\supset F^m\supset F^{m+1}=0\]
such that $F^{\bullet}$ induces a pure $\bq$-Hodge structure of weight $k$ on the graded piece $Gr_k^W:=W_k/W_{k-1}$.

One can also define a mixed Hodge structure on compactly supported cohomology $H_c^j(X,\bq)$ and the forgetful map 
\[H_c^j(X,\bq)\ra H^j(X,\bq)\]
is compatible with mixed Hodge structures (see \cite[Theorem 2.1.3]{HRV}).
 
\begin{defn}[Definition 2.1.4 in \cite{HRV}]\label{dmhp}

Define the \textbf{mixed Hodge numbers}  by 
$$h^{p,q;j}(X):=\dim_{\bc}\left(Gr_p^F\left(Gr^W_{p+q}H^j(X)\otimes\bc\right)\right),$$
and the \textbf{compactly supported mixed Hodge numbers} by
$$h_c^{p,q;j}(X):=\dim_{\bc}\left(Gr_p^F\left(Gr^W_{p+q}H_c^j(X)\otimes\bc\right)\right).$$

Form the \textbf{mixed Hodge polynomial}:
$$H(X;x,y,t):=\sum h^{p,q;j}(X)x^py^qt^j,$$
the \textbf{compactly supported mixed Hodge polynomial}:
$$H_c(X;x,y,t):=\sum h_c^{p,q;j}(X)x^py^qt^j,$$
 and the \textbf{E-polynomial}.
$$E(X;x,y):=H_c(X;x,y,-1).$$
\end{defn}
\begin{defn}\label{dvhb}
Define the \textbf{virtual Hodge numbers} by
$$h^{p,q}_v(X):=\sum_j(-1)^jh_c^{p,q;j}(X).$$
and the \textbf{virtual Betti numbers} by
$$b^v_{i}(X):=\sum_{p+q=i}h^{p,q}_v(X).$$
Define $P_v(X;z):=\sum b^v_{i}(X)z^i=E(X;z,z)=H_c(X,z,z,-1)$ to be the \textbf{virtual Poincar\'e polynomial}.
\end{defn}
\begin{rem}\label{mhtsc}
If $X$ is complete, then $H_c^j(X,\bq)\cong H^j(X,\bq)$.  If moreover $X$ is smooth, then $W_{k}H^j(X,\bq)=\left\{\begin{array}{ll}0, &k< j\\H^j(X,\bq),&k\geq j\end{array}\right.$ and $H^j(X,\bq)$ has a pure $\bq$-Hodge structure of weight $j$.
Therefore the virtual Poincar\'e polynomial $P_v(X;z)$ coincides with the ordinary Poincar\'e polynomial $P(X;z)$. 
\end{rem}
By Deligne's mixed Hodge theory (\cite[Proposition 8.3.9]{Del2} or \cite[Appendix]{HRV}), the function $[X]\mapsto E(X;x,y)$ ($[X]\mapsto P_v(X;z)$, resp.) is a motivic measure on $K(Var_{\bc})$ taking values in $\bz[x,y]$ ($\bz[z]$, resp.).  
\begin{rem}Let $X$ be smooth and complete.  Let $Y\hookrightarrow X$ be a smooth closed subvariety with $U=X\setminus Y$.  Then we have the following long exact sequence 
\begin{equation}\label{cshe}\cdots \ra H^k_c(U)\ra H^k(X)\ra H^k(Y)\ra H^{k+1}_c(U)\ra\cdots.\end{equation}
which is an exact sequence of mixed Hodge structure by Deligne's result (\cite[Proposition 8.3.9]{Del2}).  This somehow explains why
taking E-polynomials is a motivic measure.
But the function $[X]\mapsto H_c(X;x,y,t)$ is not a motivic measure because the long exact sequence (\ref{cshe}) in general does not break into short exact sequences. 
\end{rem}  

On can compute that $E(\mathbb{A}^1;x,y)=E(\mathbb{P}^1;x,y)-E(\{pt\};x,y)=xy$ and $P_v(\mathbb{A}^1;z)=z^2$.  Therefore the two motivic measures $[X]\mapsto E(X;x,y)$ and $[X]\mapsto P_v(X;z)$ can be extended to $\widehat{K}(Var_{\bc})$.  We have 
\begin{equation}E(\bl;x,y)=xy,E(\bl^i-1;x,y)=x^iy^i-1.
\end{equation}
Therefore for $[\mm]\in\widehat{K}(Var_{\bc})$, $E(\mm;x,y)\in\bz[x,y][(xy)^{-1},(x^iy^i-1)^{-1}:i\geq 1]$ and $P_v(\mm;z)\in\bz[z][z^{-2},(z^{2i}-1)^{-1}:i\geq 1]$. 

\section{Some stacks and the estimate of their dimensions.}
Our strategy to prove the main theorem is to use the motivic measure over $\widehat{K}(Var_{\bc})$.  In this section, we prove some technic results in dimension estimate in a more general setting, which we will specialize to the $\p^2$ case in \S4 to prove the main theorem.

\subsection{Some stacks.} 
Let $S$ be a Fano surface over $\bc$.  Let $L$ be a non-trivial effective line bundle on $S$ such that $\ls$ contains an integral curve.  Hence by Remark \ref{dimp} we have $h^0(L)=\chi(L)$ and $H^1(L)=H^2(L)=0$.

We define some stacks as follows.  Our notation is consistent to \cite{Yuan4} except for $\ts^a(L,\chi)$. 
\begin{defn}\label{ff}Given two integers $\chi$ and $a$, 
let $\mm_{\bullet}^a(L,\chi)$ be the (Artin) stack parametrizing pure 1-dimensional sheaves $\mf$ on $S$ with rank 0, $\det(\mf)=L$ and $\chi(\mf)=\chi$, satisfying either of the following two conditions.
 
($C_1$) $\forall \mf'\subset \mf$, $\chi(\mf')\leq a$;

($C_2$) $\mf$ is semistable.
\end{defn}
\begin{defn}Let $\mm^{ss}(L,\chi)$ ($\mm(L,\chi)$, resp.) be the substack of $\mm_{\bullet}^a(L,\chi)$ parametrizing semistable (stable, resp.) sheaves in $\mm_{\bullet}^a(L,\chi)$.  

Let $\mn(L,\chi)$ be the substack of $\mm_{\bullet}^a(L,\chi)$ parametrizing sheaves in $\mm^a_{\bullet}(L,\chi)$ with integral supports. 
   
 Let $\ts^a(L,\chi):=\mm_{\bullet}^a(L,\chi)\setminus\mn(L,\chi)$. 
\end{defn}
\begin{defn}
Let $M(L,\chi)$ be the (coarse) moduli space of $\mm(L,\chi)$.
 
Let $N(L,\chi)$ be the image of $\mn(L,\chi)$ in the (coarse) moduli space $M(L,\chi)$.  
\end{defn} 

\begin{defn}Let $\mh^m$ be the stack associated to the Hilbert scheme $S^{[m]}$ of $m$-points on $S$, viewing as the moduli space of ideal sheaves of colength $m$. \end{defn}

\begin{rem}\label{con}\begin{itemize}\item[(1)] $\mm^a_{\bullet}(L,\chi)$ is an algebraic stack of finite type with affine stabilizers.  Actually $\mm^a_{\bullet}(L,\chi)\cong[Q^a/GL_{\nu}]$ for some $\bc$-scheme $Q^a$ of finite type and some integer $\nu$.
Hence \begin{equation}\label{mma}[\mm^a_{\bullet}(L,\chi)]=\frac{[Q^a]}{[GL_{\nu}]}\in\widehat{K}(Var_{\bc}).\end{equation}

\item[(2)] In $\widehat{K}(Var_{\bc})$, we have \begin{equation}\label{mmm}[\mm(L,\chi)]=\frac{[M(L,\chi)]}{\bl-1},~[\mn(L,\chi)]=\frac{[N(L,\chi)]}{\bl-1},~[\mh^m]=\frac{[S^{[m]}]}{\bl-1}.\end{equation}

\item[(3)] We have the Hilbert-Chow morphism
\[\pi^a:\mm_{\bullet}^a(L,\chi)\ra\ls,\quad \mf\mapsto\text{supp}(\mf).\]
By \cite[Corollary 1.3]{Yuan9}, the fiber of $\pi^a$ over any curve $C\in\ls$ has dimension $g_L-1$.

\end{itemize}
\end{rem}

\begin{defn}For two integers $k>0$ and $i$, we define $\mm_{k,i}^a(L,\chi)$ to be the (locally closed) substack of $\mm_{\bullet}^a(L,\chi)$ parametrizing sheaves $\mf\in\mm_{\bullet}^a(L,\chi)$ with $h^1(\mf(-iK_S))=k$ and $h^1(\mf(-nK_S))=0,\forall n>i.$  

\end{defn}

\begin{defn}For two integers $l>0$ and $j$, we define $\mw_{l,j}^a(L,\chi)$ to be the (locally closed) substack of $\mm_{\bullet}^a(L,\chi)$ parametrizing sheaves $\mf\in\mm_{\bullet}^a(L,\chi)$ with  $h^0(\mf(-jK_X))=l$ and $h^0(\mf(-nK_X))=0,\forall n<j$. 

Let $\mv_{l,j}(L,\chi)=\mn(L,\chi)\cap \mw_{l,j}^a(L,\chi).$  Let $V_{l,j}(L,\chi)$ be the image of $\mv_{l,j}(L,\chi)$ in $N(L,\chi)$.

\end{defn} 
\begin{rem}\label{mwa}Let $\mf\in\mm^a_{\bullet}(L,\chi)$.  If $h^1(\mf(-i_0K_S))=k_0\neq0$ ($h^0(\mf(-j_0K_S))=l_0\neq0$, resp.), then $\mf\in\mm^a_{k_0,i_0}(L,\chi)\cup\bigcup_{i> i_0}\mm^a_{k,i}(L,\chi)$ ($\mf\in\mw^a_{l_0,j_0}(L,\chi)\cup\bigcup_{j< j_0}\mw^a_{l,j}(L,\chi)$, resp.)
\end{rem}
\begin{rem}\label{duke2}
(i) The map $\mf\mapsto\mf(rK_S)$ gives two inclusions 
$$\mm_{k,i}^a(L,\chi)\hookrightarrow\mm^{a+|rL.K_S|}_{k,i+r}(L,\chi+rL.K_S),$$ 
$$\mw_{l,j}^a(L,\chi)\hookrightarrow\mw^{a+|rL.K_S|}_{l,j+r}(L,\chi+rL.K_S),$$
which identify $\mn_{k,i}(L,\chi)$ with $\mn_{k,i+r}(L,\chi+rL.K_S)$ and $\mv_{l,j}(L,\chi)$ with $\mv_{l,j+r}(L,\chi+rL.K_S)$ respectively.

(ii) The map $\mf\mapsto\mf^D:=\me xt^1(\mf,K_S)$ gives an isomorphism $$\mm_{k,i}^a(L,\chi)\xrightarrow{\cong}\mw^{-\chi+a}_{k,-i}(L,-\chi),$$ which identifies $\mn_{k,i}(L,\chi)$ with $\mv_{k,-i}(L,-\chi)$.  This is because for any purely 1-dimensional sheaf $\mf$, $\mf^D$ is purely 1-dimensional, $\chi(\mf)=-\chi(\mf^D)$ and $(\mf^{D})^D\cong\mf$ (see \cite[Lemma A.4 (2)]{Yuan5}).  Since $\mathcal{E}xt^i(\mf,K_S)=0$ for $i\neq 1$, $H^1(\mf)^{\vee}\cong\Ext^1(\mf,K_S)\cong H^0(\mf^D)$ by the spectral sequence and Serre duality.

(iii) For fixed $(\chi,a)$, $\mm_{k,i}^a(L,\chi)$ ($\mw_{l,j}^a(L,\chi)$, resp.) is empty except for finitely many pairs $(k,i)$ ($(l,j)$, resp.).  In particular, $h^0(\mf)$ ($h^1(\mf)$, resp.) is nonzero if $\chi>0$ ($\chi<0$, resp.).  Hence 
\[\mm_{k,-1}^a(L,\chi)~(\mw_{l,1}^a(L,\chi),\text{ resp.})=\emptyset,\text{ if }\chi<0~(\chi>0,\text{ resp.}).\] 
By (i) we have
\[\mm_{k,i}^a(L,\chi)=\emptyset,\text{ if }\chi-(i+1)L.K_X<0,\]
and 
\[\mw_{l,1}^a(L,\chi)=\emptyset,\text{ if }\chi-(j-1)L.K_X>0.\]

\end{rem}

\begin{defn}
Let $\mh^{m,l}_L~(0\leq l\leq h^0(L+K_S))$ be the substack of $\mh^m$ parametrizing ideal sheaves $I_{m}$ of colength $m$ satisfying that $h^0(I_m(L+K_S))=l$.  

Define $\mh_L^{m,l,k}~(k\geq 0)$ to be the substack of $\mh_L^{m,l}$ parametrizing ideal sheaves $I_{m}\in\mh_L^{m,l}$ such that $h^1(I_m(L))=k$.

Let $H^{m,l}_L$ ($H^{m,l,k}_L$, resp.) be the image of $\mh^{m,l}_L$ ($\mh^{m,l,k}_L$, resp.) in $S^{[m]}$.

\end{defn}

Our strategy is to relate $\mm(L,\chi)$ to $\mh^n$ with $n=g_L-1-\chi$.  The general idea is as follows. 

Let $\chi<0$, then $H^1(\mf)\neq0 ,\forall \mf\in \mm_{\bullet}^a(L,\chi)$.  Hence we have a following exact sequence
\begin{equation}\label{MtoH}0\ra K_S\ra \widetilde{I}\ra\mf\ra 0.
\end{equation} 
If $\widetilde{I}$ in (\ref{MtoH}) is torison free, then $\widetilde{I}\cong I_{n}(L+K_S)$ for some $I_{n}\in\mh^{n}$.   In particular, 
\begin{equation}\label{mhc}h^0(I_n(L+K_S))=h^0(\widetilde{I})=h^0(\mf),~~ h^1(I_n(L))=h^1(\widetilde{I}(-K_S))=h^1(\mf(-K_S)).\end{equation}
Hence $I_n\in\mh^{n,l,k}$ with $l=h^0(\mf),k=h^1(\mf(-K_S))$.  If moreover $\mf\in\mn(L,\chi)$, then $\widetilde{I}$ is torsion free iff (\ref{MtoH}) does not split.  

On the other hand, let $\chi\geq L.K_S$.  Then for every $I_{n}\in\mh^{n}$ we have 
$$h^0(I_{n}(L))\geq h^0(L)-n=\chi(L)-n=\chi+L.K_S+1>0.$$ 
Hence $H^0(I_n(L))\neq 0$ and every nonzero element in $H^0(I_{n}(L))$ induces a following exact sequence
\begin{equation}\label{HtoM}0\ra K_S\ra I_{n}(L+K_S)\ra\mf'\ra 0.
\end{equation}
Because $\Ext^1(T,K_S)\cong H^1(T)^{\vee}=0$ for $T$ a 0-dimensional sheaf, every 0-dimensional subsheaf of $\mf'$ in (\ref{HtoM}) is a subsheaf of $I_{n}(L+K_S)$.  Hence $\mf'$ is purely 1-dimensional and $\mf'\in\mm_{\bullet}^a(L,\chi)$ for some suitable $a$.

\begin{defn}Let $\mext^1(\mm^a_{\bullet}(L,\chi),K_S)^*$ be the stack over $\mm^a_{\bullet}(L,\chi)$ parametrizing non-split extensions in $\Ext^1(\mf,K_S)$ with $\mf\in\mm^a_{\bullet}(L,\chi)$.  

We have analogous definitions for $\mext^1(\mm(L,\chi),K_S)^*$, $\mext^1(\mn(L,\chi),K_S)^*$, etc.  
\end{defn}
\begin{defn}Let $\bh^0(\mh^{m}(L))^{*}$ be the stack over $\mh^{m}$ parametrizing non-zero sections in $H^0(I_{m}(L))\cong\Hom(K_S,I_m(L+K_S))$ with $I_{m}\in \mh^{m}$. 

We have analogous definitions for $\bh^0(\mh_L^{n,l}(L))^{*}$, $\bh^0(\mh^{n}(L+K_S))^{*}$, etc.
\end{defn}

\begin{rem}\label{deeh}(i) Let $\mw^a_{l,j}(L,\chi)$ not be empty.  Then $\forall \mf\in\mw^a_{l,j}(L,\chi)$, 
$$\dim\Ext^1(\mf,K_S)=h^1(\mf)=\left\{\begin{array}{ll}-\chi,&\forall~ j>0\\l-\chi, &\text{if }j=0\end{array}\right..$$
Therefore in $\widehat{K}(Var_{\bc})$ we have 
\begin{equation}\label{dee}[\mext^1(\mw^a_{l,j}(L,\chi),K_S)^*]=\begin{cases}(\bl^{-\chi}-1)[\mw^a_{l,j}(L,\chi)],~~\forall~ j>0\\ (\bl^{l-\chi}-1)[\mw^a_{l,j}(L,\chi)], \text{ if }j=0\end{cases}.\end{equation}
Analogously we have
\begin{equation}\label{de0}[\mext^1(\mm^a_{k,i}(L,\chi),K_S)^*]=\left\{\begin{array}{ll}0,& i<0\\ (\bl^{k}-1)[\mm^a_{k,i}(L,\chi)], &\text{if }i=0\end{array}\right..\end{equation}

(ii) For any $I_m\in\mh^{m,l,k}$, we have by definition $h^0(I_m(L+K_S))=l$ and $h^0(I_m(L))=\chi(I_m(L))+h^1(I_m(L))=\chi(L)-m+k$.  Therefore if $\mh^{m,l,k}$ is not empty, in $\widehat{K}(Var_{\bc})$ we have 
\begin{equation}\label{deh}[\bh^0(\mh_L^{m,l,k}(L))^{*}]=(\bl^{\chi(L)-m+k}-1)[\mh_L^{m,l,k}],[\bh^0(\mh_L^{m,l,k}(L+K_S))^{*}]=(\bl^{l}-1)[\mh_L^{m,l,k}].\end{equation}
\end{rem}

Let $n=g_L-1-\chi$ with $L.K_S\leq \chi<0$, then we have a rational map
\begin{equation}\label{Psi}\Psi:\mext^1(\mm^a_{\bullet}(L,\chi),K_S)^*\dashrightarrow \bh^0(\mh^{n}(L))^{*},\end{equation}
which is surjective for $a$ large enough.
In particular $\Psi$ induces an injection
\[\mext^1(\mn(L,\chi),K_S)^*\hookrightarrow \bh^0(\mh^{n}(L))^{*}.\]

\subsection{The dimension estimate.}
The dimension of a quotient stack $[X/G]$ is by definition $\dim X-\dim G$.  Hence
\begin{equation}\label{dim1}\dim\mm(L,\chi)=L^2,~~\dim\mh^{m}=2m-1.\end{equation} 
By Remark \ref{dimp} we have
\begin{equation}\label{dim2}\dim\mm(L,\chi)=\dim\mn(L,\chi)=L^2, ~H^1(L)=0,\chi(L)=h^0(L).\end{equation} 
By \cite[Corollary 1.3]{Yuan9} we have
\begin{equation}\label{dim3}\dim\mm^a_{\bullet}(L,\chi)=L^2,~\dim\ts^a(L,\chi)\leq\dim\mm^a_{\bullet}(L,\chi)-\rho_L=L^2-\rho_L,\end{equation} 
where $\rho_L=\dim\ls-\dim(\ls\setminus\ls^{int})$ is the codimension of the complement of $\ls^{int}$ inside $\ls$.

\begin{rem}\label{rin}If we remove the condition that $\ls^{int}\neq \emptyset$, then by \cite[Theorem 1.2]{Yuan9} $\dim\mm^a_{\bullet}(L,\chi)\leq \dim\ls+g_L-1.$ 
\end{rem}
The following proposition improves the statements of Proposition 6.5, Lemma 6.9 and Lemma 6.10 in \cite{Yuan4}.
\begin{prop}\label{dnki2}Let $\chi$ be any integer and $m$ any positive integer.  

(1) $\dim\mv_{l,j}(L,\chi)\leq L^2+(\chi-jK_S.L)-l$ if $\mv_{l,j}(L,\chi)$ is not empty.  

(2) $\dim\mn_{k,i}(L,\chi)\leq L^2-(\chi-iK_X.L)-k$ if $\mn_{k,i}(L,\chi)$ is not empty.

(3) For $l>(\chi-jK_S.L)\geq0$, $\dim\mv_{l,j}(L,\chi)\leq L^2-\min\{l,\chi-(j+1)K_S.L\}$ if  $\mv_{l,j}(L,\chi)$ is not empty.  

(4) For $-k<(\chi-iK_S.L)\leq0$, $\dim\mn_{k,i}(L,\chi)\leq L^2+\max\{-k,\chi-(i-1)K_S.L\}$ if $\mn_{k,j}(L,\chi)$ is not empty.

(5) If $h^0(L\otimes K_S)=0$ or $L+K_S\cong \mo_S$ , then $\mh^m=\mh^{m,0}_L$.  If $L+K_S$ is non-trivially effective, assume moreover $|L+K_S|^{int}\neq\emptyset$, then for $l\geq 1$, $\dim\mh^{m,l}_L\leq m+g_L-2$ and moreover for $l\geq 2$,
$$\dim\mh_L^{m,l}\leq\max\{m+g_L-2-\rho_{L+K_S},2g_L-4\}.$$

(6) Let $\chi(L)-m> 0$.  Then for $k\geq1$ and $\mh^{m,0,k}_L$ not empty,  we have
\[\dim\mh_L^{m,0,k}\leq 2m-1-k-\min\{\rho_L,\chi(L)-m\}.\]
\end{prop}
\begin{proof}By Remark \ref{duke2} (ii), (1) ((3), resp.) is equivalent to (2) ((4), resp.).  By Remark \ref{duke2} (i), to prove (2) it is enough to show 
$$\dim\mn_{k,0}(L,\chi)\leq L^2-\chi-k$$ for any $\chi$ if $\mn_{k,0}(L,\chi)$ is not empty.  With no loss of generality, we may assume $\chi+k> 0$.

Let $n=g_L-1-\chi$.  By definition any sheaf $\mf\in\mn_{k,0}(L,\chi)$ satisfies $h^1(\mf)=k>0$ and $h^1(\mf(-K_S))=0$.  Then we have an injection
\[\mext^1(\mn_{k,0}(L,\chi),K_S)^*\hookrightarrow \bh^0(\mh^{n}(L))^{*},\]
whose image is contained in $\bh^0(\mh_L^{n,\chi+k,0}(L))^{*}$ by (\ref{mhc}).

By (\ref{deh}) we have
\begin{eqnarray}\dim\bh^0(\mh_L^{n,\chi+k,0}(L))^{*}&\leq& \chi(L)-n+2n-1\nonumber\\ &=&\chi(L)+g_L-\chi-2=L^2-\chi.\nonumber\end{eqnarray}  
Hence by (\ref{de0})
\begin{eqnarray}\label{nh0}\dim\mext^1(\mn_{k,0}(L,\chi),K_S)^*&=&k+\dim\mn_{k,0}(L,\chi_0)\nonumber\\
&\leq& \dim\bh^0(\mh_L^{n,\chi+k,0}(L))^{*}\leq L^2-\chi.\end{eqnarray}
Therefore we have $\dim\mn_{k,0}(L,\chi)\leq L^2-\chi-k$.  This proves (2).

If $l>(\chi-jK_S.L)\geq0$, then for any $\mf\in\mv_{l,j}(L,\chi)$ we have $h^1(\mf(-jK_S.L))=l-(\chi-jK_S.L)=:k_0$.  
Therefore by Remark \ref{mwa}, $\mf\in \mn_{k_0,j}(L,\chi)\cup\bigcup_{i> j}\mn_{k,i}(L,\chi)$.
Hence by (2) we have 
\begin{eqnarray}\dim\mv_{l,j}(L,\chi)&\leq& \max\{\dim\mn_{k_0,j}(L,\chi),\dim \bigcup_{i> j}\mn_{k,i}(L,\chi)\}\nonumber\\ &\leq& \max\{L^2-l, L^2-\chi+(j+1)K_S.L)\}\nonumber\\&=&L^2-\min\{l,\chi-(j+1)K_S.L)\}.\nonumber\end{eqnarray}
This proves (3).

The first statement of (5) is obvious.  Let $L+K_S$ be non-trivially effective with $|L+K_S|^{int}\neq\emptyset$.  
By definition any sheaf $I_m\in\mh_L^{m,l}$ satisfies $h^0(I_m(L+K_S))=l$.

Let $\chi':=\chi(L\otimes K_S^{\otimes 2})-m-1=\frac{(L+2K_S)(L+K_S)}2-m$.  Any sheaf $I_m\in\mh_L^{m,l}$ with $l>0$ lies in the following exact sequence 
\begin{equation}\label{kk}0\ra K_S\ra I_m(L+2K_S)\ra \mf'\ra0\end{equation}
with $h^0(\mf'(-K_S))=h^0(I_m(L+K_S))-1=l-1$. Define 
$$\widetilde{\mw^a}(L+K_S,\chi'):=\{\mf'\in\mm^{a}_{\bullet}(L+K_S,\chi')\big|h^0(\mf'(-K_S))=l-1\}.$$
Then (\ref{kk}) induces a rational map
\[\mext^1(\widetilde{\mw^a}(L+K_S,\chi'),K_S)^*\dashrightarrow \bh^0(\mh_L^{m,l}(L+K_S))^{*},\]
which is surjective for $a$ large enough.  Since $-K_S$ is ample hence effective, $h^0(\mf')\leq h^0(\mf'(-K_S))=l-1$.  Hence for all $\mf'\in\widetilde{\mw^a}(L+K_S,\chi')$, $\dim\Ext^1(\mf',K_S)=h^1(\mf')\leq l-1-\chi'$ and we have 
\[\dim\mext^1(\widetilde{\mw^a}(L+K_S,\chi'),K_S)^*\leq l-1-\chi'+\dim\widetilde{\mw^a}(L+K_S,\chi').\]
Hence 
\begin{eqnarray}\dim\bh^0(\mh_L^{m,l}(L+K_S))^{*}
&\leq& \dim\mext^1(\widetilde{\mw^a}(L+K_S,\chi'),K_S)^*\nonumber\\
&\leq& l-1-\chi'+\dim\widetilde{\mw^a}(L+K_S,\chi').\nonumber\end{eqnarray}
On the other hand by (\ref{deh})
\[\dim\bh^0(\mh_L^{m,l}(L+K_S))^{*}=l+\dim\mh_L^{m,l}\]

Therefore \begin{equation}\label{add1}\dim\mh_L^{m,l}\leq \dim\widetilde{\mw^a}(L+K_S,\chi')-\chi'-1.\end{equation}

Since $L+K_S$ is non-trivially effective with $|L+K_S|^{int}\neq\emptyset$, we have $\dim\widetilde{\mw^a}(L+K_S,\chi')\leq \dim\mm^{a}_{\bullet}(L+K_S,\chi')=(L+K_S)^2$ and hence by (\ref{add1}) we have
\begin{eqnarray}\dim\mh_L^{m,l}&\leq& (L+K_S)^2-\chi'-1\nonumber\\&=&(L+K_S)^2-\frac{(L+2K_S).(L+K_S)}2+m-1\nonumber\\ &=&m+\frac{(L+K_S).L}2-1=m+g_L-2.\nonumber\end{eqnarray} 
 
If $l\geq 2$, then by Remark \ref{mwa} we have
$\widetilde{\mw^a}(L+K_S,\chi')\subset\bigcup_{j\leq1}\mw^a_{t,j}(L+K_S,\chi').$
Hence 
$$\dim\widetilde{\mw^a}(L+K_S,\chi')\leq\max\{\dim\ts^a(L+K_S,\chi),\dim\bigcup_{j\leq1}\mv^a_{t,j}(L+K_S,\chi')\}.$$ 
Hence by (1) we have 
\begin{eqnarray}\label{add2}\dim\widetilde{\mw^a}(L+K_S,\chi')&\leq&\max\{\dim\ts^a(L+K_S,\chi),(L+K_S)^2+(\chi'-K_S.(L+K_S))-1\}\nonumber\\&=&\max\{(L+K_S)^2-\rho_{L+K_S},\chi'+L.(L+K_S)-1\}.\end{eqnarray}
By definition $\chi'=\frac{(L+2K_S).(L+K_S)}2-m$.  Therefore by (\ref{add1}) and (\ref{add2}) we have
\begin{eqnarray}\dim\mh_L^{m,l}&\leq& \dim\widetilde{\mw^a}(L+K_S,\chi')-\chi'-1
\nonumber \\&\leq&\max\{m+g_L-2-\rho_{L+K_S},2g_L-4\}.\nonumber\end{eqnarray}
 This proves (5).
 
Finally we assume $\chi(L)-m> 0$.  Since $L$ is nontrivially effective, $H^0(L^{-1})^{\vee}=H^2(L\otimes K_S)=0$ and hence $H^2(I_m(L+K_X))=0$ for any $I_m\in \mh^m$.  If $\chi(I_m(L+K_X))=\chi(L\otimes K_X)-m>0$, then $H^0(I_m(L+K_X))\neq 0$ for all $I_m\in\mh^m$ and $\mh^{m,0}$ is empty.  Hence we assume moreover $\chi(L\otimes K_X)-m\leq0$.  

Since $\chi(I_m(L))=\chi(L)-m>0$, any sheaf $I_m\in\mh_L^{m,0,k}$ lies in the following exact sequence 
\begin{equation}\label{kk2}0\ra K_S\ra I_m(L+K_S)\ra \mf''\ra0\end{equation}
with $h^0(\mf'')=h^0(I_m(L+K_S))=0,~h^1(\mf''(-K_S))=h^1(I_m(L))=k$ and $\chi(\mf'')=\chi'':=\chi(L\otimes K_S)-m-1$. Let 
$$\widetilde{\mm^a}(L,\chi''):=\left\{\mf''\in\mm^{a}_{\bullet}(L,\chi'')\left|\begin{array}{c}h^0(\mf'')=0,\\ h^1(\mf''(-K_S))=k\end{array}\right.\right\}.$$
Then (\ref{kk2}) induces a rational map
\[\mext^1(\widetilde{\mm^a}(L,\chi''),K_S)^*\dashrightarrow \bh^0(\mh_L^{m,0,k}(L))^{*},\]
which is surjective for $a$ large enough.  Hence
 \begin{equation}\dim\bh^0(\mh_L^{n,0,k}(L))^{*}\leq \dim\mext^1(\widetilde{\mm^a}(L,\chi''),K_S)^*= -\chi''+\dim\widetilde{\mm^a}(L+K_S,\chi'),\end{equation}
 where the last equality is because for every $\mf''\in \widetilde{\mm^a}(L,\chi'')$, $\dim\Ext^1(\mf'',K_S)=h^1(\mf'')=-\chi(\mf'')=-\chi''$.
 
On the other hand by (\ref{deh})
$$\dim\bh^0(\mh_L^{m,0,k}(L))^{*}=k+\chi(L)-m+\dim\mh_L^{m,0,k}.$$
Therefore \begin{equation}\label{add3}\dim\mh_L^{m,0,k}\leq \dim\widetilde{\mm^a}(L,\chi'')-\chi''-\chi(L)+m-k.\end{equation}

By Remark \ref{mwa} we have
$\widetilde{\mm^a}(L,\chi'')\subset\bigcup_{j\geq1}\mm^a_{t,j}(L,\chi'').$
Hence $$\dim\widetilde{\mm^a}(L,\chi'')\leq\max\{\dim\ts^a(L,\chi),\dim \bigcup_{j\geq1}\mn^a_{t,j}(L,\chi'')\}.$$
Hence by (2) we have
\begin{eqnarray}\label{add4}\dim\widetilde{\mm^a}(L,\chi'')&\leq&
\max\{\dim\ts^a(L,\chi),L^2-(\chi''-K_S.L)-1\}\nonumber\\&=&\max\{L^2-\rho_L,L^2+K_S.L-\chi''-1\}.\end{eqnarray}  
By definition $\chi''=\frac{L.(L+K_S)}2-m$.  Therefore by (\ref{add3}) and (\ref{add4}) we have
\begin{eqnarray}\dim\mh_L^{m,0,k}&\leq&\max\{2m-k-1-\rho_L,3m-k-1-\chi(L)\}\nonumber\\&=& 2m-1-k-\min\{\rho_L,\chi(L)-m\}.\nonumber\end{eqnarray}
 This proves (6).
\end{proof}

\begin{rem}\label{ntan}Analogous statement to Proposition \ref{dnki2} (1) ((2), resp.) may fail for $\mw^a_{l,j}(L,\chi)$ ($\mm^a_{k,i}(L,\chi)$, resp.).  For instance let $S=\p^2$ with $H$ the hyperplane class, let $\mr(1,d-1;1)(\chi)$ consist of sheaves $\mf$ lying in sequences of the following form
\[0\ra\mo_C\ra\mf\ra\mf_1\ra0,\]
where $\p^1\cong C\in |H|$ and $\mf_1\in \mn((d-1)H,\chi-1)$.  Then $\dim\mr(1,d-1;1)(\chi)=d^2-(d-1)$ by Proposition \ref{clor} and hence 
$$\dim\bigcup_{l\leq 0}\mw^a_{l,j}(d,\chi)\geq d^2-(d-1)$$ 
for any $\chi$ and $a$ large enough, while by Proposition \ref{dnki2} (1) we have $$\dim\bigcup_{l\leq 0}\mv_{l,j}(d,\chi)\leq d^2+\chi-1.$$
\end{rem}

\subsection{Sheaves supported on curves with two integral components.}
The rational map in (\ref{Psi})
\begin{equation}\label{Psi3}\Psi:\mext^1(\mm^a_{\bullet}(L,\chi),K_S)^*\dashrightarrow \bh^0(\mh^{n}(L))^{*}\end{equation}
 is surjective for $a$ large enough.  $\Psi$ restricted to $\mext^1(\mn(L,\chi),K_S)^*$ is an injection.  But restricted to $\mext^1(\ts^a(L,\chi),K_S)$ with $\ts^a(L,\chi)=\mm^a_{\bullet}(L,\chi)\setminus\mn(L,\chi)$, $\Psi$ is in general very complicated to describe.
 
In this subsection, we study the subscheme $\mr^a(L,\chi)$ of $\ts^a(L,\chi)$ consisting of sheaves whose supports only have two integral components. 

Let $L_1,L_2$ be two non-trivially effective line bundles such that $|L_i|^{int}\neq\emptyset,~i=1,2.$. Denote by $\mr^a(L_1,L_2;\chi_1)(\chi)$ the substack of $\ts^a(L_1+L_2,\chi)$ consisting of sheaves $\mf$ lying in sequences of the following form
\begin{equation}\label{Rex}0\ra\mf_1\ra\mf\ra\mf_2\ra0,\end{equation}
where $\mf_1\in\mn(L_1,\chi_1)$, $\mf_2\in\mn(L_2,\chi-\chi_1)$ and supp$(\mf_1)\neq\text{supp}(\mf_2)$.

The stack $\mr^a(L,\chi)$ is the union of all $\mr^a(L_1,L_2;\chi_1)(\chi)$ with $L_1+L_2=L$ and $|L_i|^{int}\neq\emptyset,~i=1,2.$  But this is not a disjoint union, since $\mf$ in (\ref{Rex}) also lies in the following sequence
\begin{equation}\label{lex}0\ra\mf_2'\ra\mf\ra\mf_1'\ra0,\end{equation}
where supp$(\mf_i')=\text{supp}(\mf_i),~i=1,2$ and $\chi_1\leq\chi(\mf_1')\leq \chi_1+L_1.L_2$.  In other words, $\mf\in \mr^a(L_2,L_1;\chi_2')(\chi)$ with $\chi_2'=\chi(\mf_2')\geq \chi-\chi_1-L_1.L_2$. 

Define \begin{equation}\label{defr}\mr^a(L_1,L_2)(\chi):=\bigcup_{\chi_1}\mr^a(L_1,L_2;\chi_1)(\chi),\end{equation}which is a disjoint union if $L_1\neq L_2$.  Obviously, $\mr^a(L_1,L_2)(\chi)=\mr^a(L_2,L_1)(\chi)$.
Notice that every $\mf\in \mr^a(L_1,L_2;\chi_1)(\chi)$ contains two subsheaves $\mf_1$ and $\mf_2'$ with $\chi(\mf_1)=\chi_1$ and $\chi(\mf_2')\geq \chi-\chi_1-L_1.L_2$.  Hence $\mr^a(L_1,L_2;\chi_1)(\chi)$ is empty unless $\chi-a-L_1.L_2\leq \chi_1\leq a$.  Therefore the union in (\ref{defr}) only contains finitely many non-empty components.

We want to describe $[\mr^a(L_1,L_2;\chi_1)(\chi)]$ in $\widehat{K}(Var_{\bc})$.  

Sheaves in $\mr^a(L_1,L_2;\chi_1)(\chi)$ may not be stable, but the following lemma says that $\Hom(\mf,\mf)\cong \bc$ for all non-decomposable $\mf\in \mr^a(L_1,L_2;\chi_1)(\chi)$.  
\begin{lemma}\label{auto}Let $\mf$ be a purely 1-dimensional sheaf with \emph{supp}$(\mf)=C_1\cup C_2$, where $C_1,C_2$ are two distinct integral curves.  Then either $\Hom(\mf,\mf)\cong \bc$ and all endomorphisms of $\mf$ are given by multiplication by scalars or $\mf\cong\mg_1\oplus\mg_2$ with $\mg_i$ a torison-free rank 1 sheaf on $C_i,i=1,2$. 
\end{lemma}
\begin{proof}Let $\widetilde{\mf}_2:=\mf\otimes\mo_{C_2}$, then $\widetilde{\mf}_2$ is a $\mo_{C_2}$-module.  However, $\widetilde{\mf}_2$ might contain non-trivial 0-dimensional subsheaves.  Denote by $\mf_2$ the quotient of $\widetilde{\mf}_2$ modulo its maximal 0-dimensional subsheaf.  Then we have the surjections $\mf\twoheadrightarrow\widetilde{\mf}_2\twoheadrightarrow\mf_2$.  We can write down an exact sequence as follows.
\begin{equation}\label{2ex}0\ra\mf_1\ra\mf\ra\mf_2\ra0,\end{equation}
where $\mf_1$ ($\mf_2$, resp.) is purely 1-dimensional sheaf with support $C_1$ ($C_2$, resp.).


Since $C_1,C_2$ are two distinct integral curves and $\mf_1,\mf_2$ are pure, we have $\Hom(\mf_1,\mf_2)=\Hom(\mf_2,\mf_1)=0$.  Therefore every endomorphism $\sigma\in \Hom(\mf,\mf)$ induces an element $(\sigma_1,\sigma_2)\in\Hom(\mf_1,\mf_1)\times\Hom(\mf_2,\mf_2)$.  Hence we have a map (of $\bc$-vector spaces) 
$$Res:\Hom(\mf,\mf)\ra\Hom(\mf_1,\mf_1)\times\Hom(\mf_2,\mf_2),~\sigma\mapsto(\sigma_1,\sigma_2).$$
If $\sigma_1=\sigma_2=0$, then $\sigma$ induces an element in $\Hom(\mf_2,\mf_1)=0$ and hence $\sigma=0$.  Thus the map $Res$ is injective.  Since $\Hom(\mf_i,\mf_i)\cong\bc,i=1,2$, it is enough to show that $Res$ is surjective iff (\ref{2ex}) splits.

If $Res$ is surjective, take $\sigma\in\Hom(\mf,\mf)$ such that $\sigma_1=id_{\mf_1},\sigma_2=0$.  Then the image of $\sigma$ is contained in $\mf_1$ and moreover $\sigma$ is a split of (\ref{2ex}).  This finishes the proof of the lemma.
\end{proof}
\begin{rem}Let $\mf$ be a purely 1-dimensional sheaf with support consisting of two distinct integral curves.  Then by Lemma \ref{auto}, $\mf$ is either simple or decomposable. 
\end{rem}
\begin{prop}\label{clor}Let $a$ be large enough, then in $\widehat{K}(Var_{\bc})$ we have
\begin{eqnarray}[\mr^a(L_1,L_2;\chi_1)(\chi)]&=&\bl^{L_1.L_2}[\mn(L_1,\chi_1)][\mn(L_2,\chi-\chi_1)]\nonumber\\ &=&\bl^{L_1.L_2}\frac{[N(L_1,\chi_1)]}{\bl-1}\frac{[N(L_2,\chi-\chi_1)]}{\bl-1}.\nonumber\end{eqnarray}
\end{prop}
\begin{proof}For every sequence 
\[0\ra\mf_1\ra\mf\ra\mf_2\ra0,\]
with $\mf\in\mr^a(L_1,L_2;\chi_1)(\chi)$, by (\ref{2ex}) we have $\Hom(\mf_1,\mf_1)\ra\Hom(\mf_1,\mf)$ is bijective and identifies $\text{Aut}(\mf_1)$ with $\Hom(\mf_1,\mf)\setminus\{0\}$.  Moreover for a fixed map $\mf_1\ra\mf$, the orbit of $\text{Aut}(\mf_2)$ in $\Ext^1(\mf_2,\mf_1)$ is exactly all extensions with middle term isomorphic to $\mf$.

Let $\mhom(\mn(L_1,\chi_1),\mr^a(L_1,L_2;\chi_1)(\chi))^*$ be the stack parametrizing all non-zero maps $\mf_1\ra\mf$ with $\mf_1\in \mn(L_1,\chi_1),\mf\in\mr^a(L_1,L_2;\chi_1)(\chi)$.  Since every non-zero map $\mf_1\ra\mf$ is injective with cokernel in $\mn(L_2,\chi-\chi_1)$, for $a$ big enough we have 
\[\mhom(\mn(L_1,\chi_1),\mr^a(L_1,L_2;\chi_1)(\chi))^*\cong\mext^1(\mn(L_2,\chi-\chi_1),\mn(L_1,\chi_1)).\]

On the other hand, for every $\mf\in \mr^a(L_1,L_2;\chi_1)(\chi)$, there exists uniquely an $\mf_1\in \mn(L_1,\chi_1)$ such that $\Hom(\mf_1,\mf)\neq 0$.  Also $\text{Aut}(\mf_1)\cong\Hom(\mf_1,\mf)\setminus\{0\}$.  Therefore we have
$\mhom(\mn(L_1,\chi_1),\mr^a(L_1,L_2;\chi_1)(\chi))^*\cong \mr^a(L_1,L_2;\chi_1)(\chi).$
The proposition follows from \[[\mext^1(\mn(L_2,\chi-\chi_1),\mn(L_1,\chi_1))]=\bl^{L_1.L_2}[\mn(L_1,\chi_1)][\mn(L_2,\chi-\chi_1)],\]
which is because $\forall~\mf_i\in\mn(L_i,\chi_i),~i=1,2,~\dim\Ext^1(\mf_2,\mf_1)=\chi(\mf_2,\mf_1)=L_1.L_2$ by Riemann-Roch.
\end{proof}

Now we study the rational map 
$\Psi$ in (\ref{Psi3}) restricted to $\mext^1(\mr^a(L,\chi),K_S)^*$.  In general $\Psi$ is not well-defined over all $\mext^1(\mr^a(L,\chi),K_S)^*$.  

Let $\mf\in \mr^a(L_1,L_2;\chi_1)(\chi)\bigcap\mr^a(L_2,L_1;\chi_2)(\chi)$, then we have two exact sequences
\[0\ra\mf_1\ra\mf\ra\mf_2\ra0,\]
\[0\ra\mf_2'\ra\mf\ra\mf_1'\ra0,\]
where $\mf_1\in\mn(L_1,\chi_1),\mf_2'\in\mn(L_2,\chi_2)$.  Since $\Hom(\mg,K_S)=0$ for any 1-dimensional sheaf $\mg$, by long exact sequences we have $\Ext^1(\mf_2,K_S)\hookrightarrow \Ext^1(\mf,K_S)$ and $\Ext^1(\mf_1',K_S)\hookrightarrow \Ext^1(\mf,K_S)$.   In fact $\Ext^1(\mf_2,K_S)$ ($\Ext^1(\mf_1',K_S)$, resp.) can be viewed as the subspace of $\Ext^1(\mf,K_S)$ parametrizing extensions that partially split along $\mf_1$ ($\mf_2'$, resp.).  We have the following lemma.  

\begin{lemma}\label{tfcr}Let $\mf,\mf_i,\mf_i',i=1,2$ be as in the previous paragraph, then

(1) $\chi-L_1.L_2\leq\chi_1+\chi_2\leq \chi$.

(2) In an extension $\eta=[0\ra K_S\ra\widetilde{I}\ra\mf\ra0]\in\Ext^1(\mf,K_S)$, $\widetilde{I}$ is torison-free iff $\eta$ is not contained in $\Ext^1(\mf_2,K_S)\cup\Ext^1(\mf_1',K_S)$.

(3) Viewed as two subspaces of $\Ext^1(\mf,K_S)$, $\Ext^1(\mf_2,K_S)\bigcap\Ext^1(\mf_1',K_S)=\{0\}$.  In particular, $h^1(\mf_2)+h^1(\mf_1')\leq h^1(\mf)$. 
\end{lemma}
\begin{proof}Let supp$(\mf)=C_1\cup C_2$ with $C_i\in |L_i|^{int},i=1,2$.  Then $\mf_2$ is the torsion-free quotient of $\mf\otimes\mo_{C_2}$ and $\mf_2'$ is the extension of the maximal 0-dimensional subsheaf of $\mf\otimes\mo_{C_1}$ by $\mf_2\otimes L_1^{-1}$.  Therefore 
\[\chi(\mf_2')+L_1.L_2\geq\chi(\mf_2)\Leftrightarrow  \chi_1+\chi_2\geq\chi-L_1.L_2.\]

Since the composition of two maps $\mf_2'\hookrightarrow\mf\twoheadrightarrow\mf_2$ can't be zero, we have
\[\chi(\mf_2')\leq\chi(\mf_2)\Leftrightarrow \chi_2+\chi_1\leq\chi.\]
This proves (1).

For any element $\eta=[0\ra K_S\ra\widetilde{I}\ra\mf\ra0]\in\Ext^1(\mf,K_S)$, the torsion of $\widetilde{I}$ has to be a subsheaf of $\mf$.  All extensions of 0-dimensional sheaves by $K_S$ are trivial, hence if $\eta$ splits along a subsheaf of $\mf_1$ ($\mf_2'$, resp.), it must split along $\mf_1$ ($\mf_2'$, resp.) and $\eta\in \Ext^1(\mf_2,K_S)$ ($\Ext^1(\mf_1',K_S)$, resp.).  
Therefore $\widetilde{I}$ is torison-free $\Leftrightarrow\eta\not\in \Ext^1(\mf_2,K_S)\cup\Ext^1(\mf_1',K_S)$.  This proves (2).

Since $\mf/(\mf_1\oplus\mf_2')$ is a 0-dimensional sheaf and all extensions of 0-dimensional sheaves by $K_S$ are trivial, the extension which splits along 
$\mf_1\oplus\mf_2'$ has to split.  This proves $\Ext^1(\mf_2,K_S)\bigcap\Ext^1(\mf_1',K_S)=\{0\}$ and hence $h^1(\mf_2)+h^1(\mf_1')\leq h^1(\mf)$ by Serre duality and an elementary fact in linear algebra.  This proves (3). 
\end{proof}

\section{Proof of the main theorem.}
In this section, we prove our main theorem: Theorem \ref{main}.  

Let $S=\p^2$ and $L=dH$.  Then $g_L=\frac{(d-1)(d-2)}2$ and $\rho_L=d-1$.  We write $\mm_{\bullet}^a(d,\chi)$ ($\mm(d,\chi)$, $\mn(d,\chi)$, $\ts^a(d,\chi)$, $\mn_{k,i}(d,\chi)$ etc, resp.) instead of $\mm_{\bullet}^a(dH,\chi)$ ($\mm(dH,\chi)$, $\mn(dH,\chi)$, $\ts^a(dH,\chi)$, $\mn_{k,i}(dH,\chi)$ etc, resp.).

We have the following proposition as a direct corollary to Proposition \ref{dnki2}.
\begin{prop}\label{dnkip}

(1) $\dim\mv_{l,j}(d,\chi)\leq d^2+(\chi+3jd)-l$ if $\mv_{l,j}(d,\chi)$ is not empty.  

(2) $\dim\mn_{k,i}(d,\chi)\leq d^2-(\chi+3di)-k$ if $\mn_{k,j}(d,\chi)$ is not empty.


(3) Let $d\geq 4$ and $n=\frac{d(d-1)}2+1=g_L+d$.  For $l\geq 1$, $\dim\mh^{n,l}_L\leq 2n-1-(d+1)$ and moreover for $l\geq 2$,
$$\dim\mh_L^{n,l}\leq 2n-1-(2d-3)=2g_L-4.$$

(3) Let $n=\frac{d(d-1)}2+1<\chi(L)=\frac{d(d+3)}2+1$.  For $k\geq1$ and $\mh^{n,0,k}_L$ not empty,  we have
\[\dim\mh_L^{n,0,k}\leq 2n-1-(d-1+k).\]
\end{prop}

From now on we always assume $d\geq 5$, $\chi=-d-1$ and $n=g_L-1-\chi=\frac{d(d-1)}2+1$.  Since $(\chi,d)=1$, $M(d,\chi)=M^{ss}(d,\chi)$ is smooth of dimension $d^2+1$.  We write $\mm_{k,i}$ instead of $\mm_{k,i}(d,\chi)$ if the invariants $d,\chi$ are already mentioned in the formula.  For instance $\mn(d,\chi)\cap \mm_{k,i}=\mn_{k,i}(d,\chi)$.  We use $\mn_{k,i}$, $\mw_{l,j}$ and $\mv_{l,j}$ analogously. 

Let $a$ be large enough so that the rational map defined in (\ref{Psi})
\begin{equation}\label{Psip}\Psi:\mext^1(\mm^a_{\bullet}(d,\chi),K_S)^*\dashrightarrow \bh^0(\mh^{n}(d))^{*},\end{equation}
is surjective.

Denote by $\widehat{K}_m$ the subgroup (with the group structure given by the addition) of $\widehat{K}(Var_{\bc})$ consisting of $[\mm]$ such that $\dim\mm\leq m$.  For $[\ma],[\mb]\in\widehat{K}(Var_{\bc})$, we write $[\ma]\equiv[\mb]~~(\widehat{K}_m)$ if $[\ma]-[\mb]\in \widehat{K}_m$.

By Proposition \ref{dnkip} (3) and (4) we have 
 \begin{equation}\label{modh1}[\mh^{n,0}_L]\equiv [\mh^n]~~(\widehat{K}_{2n-1-(2d-3)}),\end{equation}
 \begin{equation}\label{modh2}[\mh^{n,0,0}_L]\equiv [\mh^n]~~(\widehat{K}_{2n-1-d}),\end{equation}
and
\begin{equation}\label{modh3}[\mh^{n,0,0}_L]\equiv [\mh^n]-[\mh_L^{n,0,1}]~~(\widehat{K}_{2n-1-(d+1)}).\end{equation}
Moreover by definition for any $I_n\in \mh^{n,0,k}_L$, we have $h^0(I_n(d-3))=0$ and $h^0(I_n(d))=\chi(I_n(d))+k=2d+k$, hence we have
\begin{eqnarray}\label{modhh}[\bh^0(\mh^{n,0,0}_L(d))^*]&=&(\bl^{2d}-1)\cdot[\mh^{n,0,0}_L(d)]\nonumber\\
&\equiv&(\bl^{2d}-1)\cdot([\mh^n]-[\mh_L^{n,0,1}])~~(\widehat{K}_{2n+d-2})\nonumber\\
&\equiv&\bl^{2d}\cdot([\mh^n]-[\mh_L^{n,0,1}])~~(\widehat{K}_{d^2})
\end{eqnarray}
The last equality in (\ref{modhh}) is because $2n-1<2n+d-2=d^2$.

Since $\Psi$ in (\ref{Psip}) is surjective, define $\bu:=\Psi^{-1}(\bh^0(\mh^{n,0,0}_L(d))^*)$.  Define the projection $\phi:\mext^1(\mm^a_{\bullet}(d,\chi),K_S)^*\ra\mm^a_{\bullet}(d,\chi)$.  For every $\mf\in\phi(\bu)$, there is an exact sequence
\[0\ra K_S\ra I_n(d-3)\ra\mf\ra0,\]
with $h^0(I_n(d-3))=h^1(I_n(d))=0$.  Hence for any $\mf\in\phi(\bu)$ we have $h^0(\mf)=h^1(\mf(3))=0$.  Therefore $\dim\Ext^1(\mf,K_S)=h^1(\mf)=-\chi=d+1,\forall \mf\in\phi(\bu)$ and we have
\begin{equation}\label{phied}[\phi^{-1}(\phi(\bu))]=(\bl^{d+1}-1)[\phi(\bu)].\end{equation}

\begin{lemma}\label{phiu}We have
$$\mn(d,\chi)\setminus \left(\bigcup _{j\leq 0}\mv_{l,j}\cup\bigcup_{i\geq 1}\mn_{k,i}\right)\subset\phi(\bu)\subset \mm^a_{\bullet}(d,\chi)\setminus \left(\bigcup _{j\leq 0}\mw^{a}_{l,j}\cup\bigcup_{i\geq 1}\mm^a_{k,i}\right).$$ \end{lemma}
\begin{proof}Since $h^0(\mf)=h^1(\mf(3))=0,~\forall\mf\in\phi(\bu)$, by Remark \ref{mwa} we have
$$\phi(\bu)\subset \mm^a_{\bullet}(d,\chi)\setminus \left(\bigcup _{j\leq 0}\mw^{a}_{l,j}\cup\bigcup_{i\geq 1}\mm^a_{k,i}\right).$$

For every $\mf \in\mm^a_{\bullet}(d,\chi)\setminus \left(\bigcup _{j\leq 0}\mw^{a}_{l,j}\cup\bigcup_{i\geq 1}\mm^a_{k,i}\right)$, $\mf\in\phi(\bu)$ iff there exists a torison-free extension of $\mf$ by $K_S\cong\mo_{\p^2}(-3)$.  Hence 
$$\mn(d,\chi)\setminus \left(\bigcup _{j\leq 0}\mv_{l,j}\cup\bigcup_{i\geq 1}\mn_{k,i}\right)\subset\phi(\bu).$$
The lemma is proved.
\end{proof}

Recall that $\mr^a(L_1,L_2;\chi_1)(\chi)$ is the substack of $\ts^a(L_1+L_2,\chi)$ consisting of sheaves $\mf$ lying in sequences of the following form
\[0\ra\mf_1\ra\mf\ra\mf_2\ra0,\]
where $\mf_1\in\mn(L_1,\chi_1),\mf_2\in\mn(L_2,\chi-\chi_1)$ and supp$(\mf_1)\neq\text{supp}(\mf_2)$.  Also $\mr^a(L_2,L_1)(\chi)=\mr^a(L_1,L_2)(\chi)=\displaystyle{\bigcup_{\chi-L_1.L_2-a\leq \chi_1\leq a}}\mr^a(L_1,L_2;\chi_1)(\chi)$

Define $\wmn^a(d,\chi):=\mn(d,\chi)\cup\mr^a(1,d-1)(\chi),~~\wmn(d,\chi):=\wmn^a(d,\chi)\cap\mm(d,\chi)$, $\wmn_{k,i}^a(d,\chi):=\wmn^a(d,\chi)\cap\mm^a_{k,i}$ and finally $\wmv_{l,j}^a(d,\chi):=\wmn^a(d,\chi)\cap\mw^a_{l,j}$.   

Denote by $\widetilde{\ls^{int}}$ the image of $\wmn^a(d,\chi)$ via the Hilbert-Chow morphism $\pi^a$.  By a direct computation we have $\dim \ls-\dim(\ls\setminus\widetilde{\ls^{int}})=2(d-2)$.  Therefore by \cite[Corollary 1.3]{Yuan9}, we have 
\begin{equation}\label{codtn}\dim\left(\mm^a_{\bullet}(d,\chi)\setminus\wmn^a(d,\chi)\right)\leq d^2-2(d-2)\leq d^2-(d+1),\end{equation} 
and hence $\dim\left(\mm(d,\chi)\setminus\wmn(d,\chi)\right)\leq d^2-(d+1).$

Define $$\um:=\mm(d,\chi)\cup\mr^a(1,d-1;-1)(\chi)\cup\mr^a(d-1,1;\chi+2)(\chi)\cup\mr^a(d-1,1;\chi+3)(\chi).$$  
This is a disjoint union because every sheaf in $\mr^a(1,d-1;-1)(\chi)$ has $\mo_H(-2)$ as a subsheaf and every sheaf in $\mr^a(d-1,1;\chi+2)(\chi)$ ($\mr^a(d-1,1;\chi+3)(\chi)$, resp.) has $\mo_H(-3)$ ($\mo_H(-4)$, resp.) as a quotient.  Here $\mo_H(s)$ stands for $\mo_C\otimes\mo_{\p^2}(s)$ with $\p^1\cong C\in |H|$.  
\begin{lemma}\label{tired}$\phi(\bu)\cap\wmn^a(d,\chi)\subset\um$.
\end{lemma}
\begin{proof}We only need to show $\phi(\bu)\cap\mr^a(1,d-1)(\chi)\subset \um$.

Let $\mf\in\phi(\bu)\cap\mr^a(1,d-1)(\chi)$ and assume $\mf$ is not stable.  Then
$\mf$ lies in either of the following two exact sequences
\begin{equation}\label{gex1}0\ra\mg_1\ra\mf\ra\mo_{H}(s')\ra0,\end{equation}
\begin{equation}\label{gex2}0\ra\mo_H(t')\ra\mf\ra\mg_2\ra0,\end{equation}
where $s'\leq -3$ and $t'\geq -2$.  

On the other hand for every $\mf\in\phi(\bu)$ we have $h^1(\mf(-K_S))=0$ which implies $H^1(\mo_H(s'+3))=0\Leftrightarrow s'\geq -4$.  We also have that $h^1(\mf)\neq h^1(\mg_2)$ since otherwise every extension of $\mf$ by $K_S$ splits along $\mo_H(t')$ and $\mf\not\in \phi(\bu)$.  This implies $H^1(\mo_H(t'))\neq0\Leftrightarrow t'\leq -2$.  Hence the lemma.
\end{proof}
Let $\ms_1(s)$ ($\mt_1(t)$, resp.)$\subset\mr^a(d-1,1;\chi-s-1)(\chi)$ ($\mr^a(1,d-1;t+1)(\chi)$, resp.) consist of all sheaves $\mf$ such that $h^0(\mf)\neq 0$. 

Let $\ms_2(s)$ ($\mt_2(t)$, resp.)$\subset\mr^a(d-1,1;\chi-s-1)(\chi)$ ($\mr^a(1,d-1;t+1)(\chi)$, resp.) consist of all sheaves $\mf$ such that $h^1(\mf(-K_S))\neq 0$.  

Let $\ms_3(s)$ ($\mt_3(t)$, resp.)$\subset\mr^a(d-1,1;\chi-s-1)(\chi)$ ($\mr^a(1,d-1;t+1)(\chi)$, resp.) consist of all sheaves $\mf$ such that there is no torison free extension of $\mf$ by $K_S$.  

Then we have 
\begin{equation}\label{rest}\mr^a(d-1,1;\chi-s-1)(\chi)\setminus\left(\bigcup_{i=1}^3\ms_i(s)\right)\subset\phi(\bu)\end{equation} and 
\begin{equation}\label{lest}\mr^a(1,d-1;t+1)(\chi)\setminus\left(\bigcup_{i=1}^3\mt_i(t)\right)\subset\phi(\bu).\end{equation}
Since $R^a(L_1,L_2)(\chi)=R^a(L_2,L_1)(\chi)$, we have for $i=1,2,3$
\begin{equation}\label{seqt}\bigcup_{s}\ms_i(s)=\bigcup_t\mt_i(t).
\end{equation}
By Lemma \ref{tfcr} (1), we have 
\begin{equation}\label{inept}\ms_i(s)\cap\mt_i(t)=\emptyset,\text{ if }t>s\text{ or }t<s-d+1.\end{equation}
Also one can easily see the following three properties  
\[\left\{\begin{array}{ll}\ms_2(s)=\mr^a(d-1,1;\chi-s-1)(\chi),&\text{ for } s\leq-5\\\mt_1(t)=\mr^a(1,d-1;t+1)(\chi),&\text{ for }t\geq 0 \\ \mt_3(t)=\mr^a(1,d-1;t+1)(\chi),&\text{ for }t\geq -1\end{array}\right..\]

\begin{lemma}\label{stlr}(1) For $s,t\leq -1$, we have
$$\dim\ms_1(s)\leq d^2-2(d+1)-s,~~\dim\mt_1(t)\leq d^2-2(d+1)-t.$$  

(2) For $s,t\geq -4$, we have 
$$\dim\ms_2(s)\leq  d^2-(3d-5)+s, ~~\dim\mt_2(t)\leq  d^2-(3d-5)+t.$$

(3) For $-d-1\leq s\leq -2$,   $\ms_3(s)=\emptyset$.

(4) For $s\geq -1,t\leq-2$, $\dim(\ms_3(s)\cap\mt_3(t))\leq  d^2-2d$.  

\end{lemma}
\begin{proof}Let $\mf\in\mr^a(d-1,1;\chi-s-1)(\chi)$.  We have an exact sequence
as follows.
\begin{equation}\label{fex1}0\ra\mg_1\ra\mf\ra\mo_{H}(s)\ra0.\end{equation}
$\mf$ also lies in the following sequence
\begin{equation}\label{fex2}0\ra\mo_H(t')\ra\mf\ra\mg_2\ra0,\end{equation}
where $1-d+s\leq t'\leq s$ by Lemma \ref{tfcr} (1).

If $s\leq -1$, then $h^0(\mf)\neq 0\Leftrightarrow h^0(\mg_1)\neq 0\Leftrightarrow \mg_1\in\bigcup_{j\leq 0}\mv_{l,j}(d-1,\chi-s-1)$. 
By Proposition \ref{dnkip} (1), we have 
$$\dim\bigcup_{j\leq 0}\mv_{l,j}(d-1,\chi-s-1)\leq (d-1)^2-1+\chi-s-1,$$
Hence by Proposition \ref{clor}, we have $\dim\ms_1(s)\leq d^2-(d+1)+\chi-s=d^2-2(d+1)-s$ since $\chi=-d-1$.

The statement for $\mt_1(t),t\leq -1$ can be proved analogously.  Hence (1) is proved.

Now we look at $\ms_2(s)$.  
If $s\geq -4$, in (\ref{fex1}) we have $h^1(\mf(-K_S))\neq 0\Rightarrow h^1(\mg_1(-K_S))\neq 0\Leftrightarrow \mg_1\in \bigcup_{i\geq 1}\mn_{k,i}(d-1,\chi-s-1)$.

By Proposition \ref{dnkip} (2), we have
\[\dim \bigcup_{i\geq 1}\mn_{k,i}(d-1,\chi-s-1)\leq (d-1)^2-3(d-1)-1-\chi+s+1,\]
Hence by Proposition \ref{clor} we have $\dim\ms_2(s)\leq  d^2-4(d-1)-\chi+s=d^2-(3d-5)+s$.

The statement for $\mt_1(t),t\geq -4$ can be proved analogously.  Hence (2) is proved.

Let $\mf\in\ms_3(s)$.  By Lemma \ref{tfcr} (2) (3), we have either $ h^1(\mo_{H}(s))=h^1(\mf)\geq -\chi=d+1$ or $h^1(\mg_2)=h^1(\mf)\Rightarrow h^1(\mo_H(s))=0$.  Hence for $-d-1\leq s\leq -2$,   $\ms_3(s)=\emptyset$.  (3) is proved.

Let $\mf\in \ms_3(s)\cap\mt_3(t')$ with $s\geq -1,t'\leq -2$.  Then by (\ref{inept}) we have $t'\geq s-d+1\geq -d$.  Since $h^1(\mo_H(s))=0$, by Lemma \ref{tfcr} (1) in (\ref{fex2}) we have $h^1(\mg_2)=h^1(\mf)\geq -\chi= d+1$.
Hence $$\mg_2\in\bigcup_{k\geq d+1}\mn_{k,0}(d-1,\chi-t-1)\cup\bigcup_{i\geq 1}\mn_{k,i}(d-1,\chi-t-1).$$

Since $-d\leq t'\leq-2\Leftrightarrow -d\leq \chi-t'-1\leq -2$, by Proposition \ref{dnki2} (3) for $-(d+1)<\chi-t'-1$, we have
\begin{eqnarray}&&\dim \bigcup_{k\geq d+1}\mn_{k,0}(d-1,\chi-t'-1)\nonumber\\&\leq& (d-1)^2+\max\{-(d+1),\chi-t'-1-3(d-1)\}=(d-1)^2-(d+1).\nonumber\end{eqnarray}

By Proposition \ref{dnkip} (2) we have
\[\dim\bigcup_{i\geq 1}\mn_{k,i}(d-1,\chi-t'-1)\leq (d-1)^2-3(d-1)-\chi+t'+1-1\leq (d-1)^2-(2d-2).\]
Hence by Proposition \ref{clor} we have $\dim(\ms_3(s)\cap \mt_3(t'))\leq  d^2-2d$ for $s\geq-1,t'\leq-2$.  Hence we proved (4). 
\end{proof}

\begin{prop}\label{msbu}In $\widehat{K}(Var_{\bc})$ we have
\[[\phi(\bu)]\equiv[\um]~~(\widehat{K}_{d^2-(d+1)}).\]
\end{prop}
\begin{proof}By (\ref{codtn}) and Lemma \ref{tired}, it is enough to prove $\dim(\um\setminus \phi(\bu))\leq d^2-(d+1)$. 

By definition 
\[\um=\mm(d,\chi)\cup\mr^a(1,d-1;-1)(\chi)\cup\mr^a(d-1,1;\chi+2)(\chi)\cup\mr^a(d-1,1;\chi+3)(\chi).\]
By (\ref{rest}) and (\ref{lest}), 
\[\um\setminus\phi(\bu)\subset\big(\mm(d,\chi)\setminus\phi(\bu)\big)\cup\bigcup_{i=1}^3\mt_i(-2)\cup\bigcup_{i=1}^3\ms_i(-3)\cup\bigcup_{i=1}^3\ms_i(-4).\]
By Lemma \ref{stlr} (1) (2) for $i=1,2$, 
$$\dim\mt_i(-2)\leq d^2-2d,\dim\ms_i(-3)\leq d^2-2d+1,\dim\ms_i(-4)\leq d^2-2d+2.$$  Also by Lemma \ref{stlr} (3), $\ms_3(-3)=\ms_3(-4)=\emptyset$.  By (\ref{seqt}) and (\ref{inept}) we have
\begin{eqnarray}\dim\mt_3(-2)&=&\dim\left(\left(\bigcup_{s}\ms_3(s)\right)\cap\mt_3(-2)\right)
\nonumber\\&=&\dim\left(\left(\bigcup_{-2\leq s\leq d-3}\ms_3(s)\right)\cap\mt_3(-2)\right)\nonumber\\ &=&\dim\bigcup_{-2\leq s\leq d-3}\big((\ms_3(s)\cap\mt_3(-2)\big)\leq d^2-2d,\nonumber\end{eqnarray}
where the last inequality is because of Lemma \ref{stlr} (4).

Now it is enough to prove $\dim(\mm(d,\chi)\setminus\phi(\bu))\leq d^2-(d+1)$.  

By (\ref{codtn}) it is enough to show $\dim(\wmn(d,\chi)\setminus\phi(\bu))\leq d^2-(d+1)$.  By Lemma \ref{phiu} and Proposition \ref{dnkip} (2), we have 
\[\dim(\mn(d,\chi)\setminus\phi(\bu))\leq\dim\left(\bigcup _{j\leq 0}\mv_{l,j}\cup\bigcup_{i\geq 1}\mn_{k,i}\right)\leq d^2-(d+2).\]  
Hence it suffices to show $\dim\big( (\mr^a(1,d-1)(\chi)\cap\mm(d,\chi))\setminus\phi(\bu)\big)\leq d^2-(d+1)$.

For every $\mf\in\mr^a(1,d-1)(\chi)\cap\mm(d,\chi)$, we have two exact sequences as follows.
\[0\ra\mg_1\ra\mf\ra\mo_{H}(s)\ra0,\]
\[0\ra\mo_H(t)\ra\mf\ra\mg_2\ra0,\]
where $s\geq -2$, $t\leq -3$ and $s-d+1\leq t\leq s$. 

By (\ref{rest}) and (\ref{lest}),
\[\big( (\mr^a(1,d-1)(\chi)\cap\mm(d,\chi))\setminus\phi(\bu)\big)\subset\bigcup_{i=1}^3\left( \bigcup_s\ms_i(s)\cap\mm(d,\chi)\right).\]

We have for $i=1,2,3$
$$\left(\bigcup_s\ms_i(s)\right)\cap\mm(d,\chi)=\left(\bigcup_t\mt_i(t)\right)\cap\mm(d,\chi)\subset\left(\bigcup_{s\geq -2}\ms_1(s)\right)\cap\left(\bigcup_{t\leq -3}\mt_1(t)\right).$$
Therefore by Lemma \ref{stlr}, we have for $i=1,2,3$
\[\left(\bigcup_s\ms_i(s)\right)\cap\mm(d,\chi)\leq d^2-(d+3).\]
Hence $\dim\big( (\mr^a(1,d-1)(\chi)\cap\mm(d,\chi))\setminus\phi(\bu)\big)\leq d^2-(d+3)$ and the proposition is proved.
\end{proof}
\begin{rem}\label{remsbu}From the proof of Proposition \ref{msbu}, we see that 
\[[\phi(\bu)]\equiv[\um]~~(\widehat{K}_{d^2-(d+2)}).\]

\end{rem}

\begin{prop}\label{stlo}In $\widehat{K}(Var_{\bc})$ we have
\[[\bu]\equiv\bl^{d+1}[\um]-\bl^{d}[\mr^a(1,d-1;-1)(\chi)]~~(\widehat{K}_{d^2}).\] 
\end{prop}
\begin{proof}By (\ref{phied}) and Proposition \ref{msbu} we have 
\[[\phi^{-1}(\phi(\bu))]=(\bl^{d+1}-1)[\phi(\bu)]\equiv(\bl^{d+1}-1)[\um]\equiv\bl^{d+1}[\um]~~(\widehat{K}_{d^2}),\]
where the last equality is because $\dim\um=d^2$.  

It is enough to show $[\phi^{-1}(\phi(\bu)\cap\um)\setminus\bu]\equiv\bl^{d}[\mr^a(1,d-1;-1)(\chi)]~~(\widehat{K}_{d^2})$.  

Obviously, $\phi^{-1}(\phi(\bu)\cap\mn(d,\chi))\subset\bu$, hence $\phi^{-1}(\phi(\bu)\cap\mn(d,\chi))\setminus\bu=\emptyset$.

Hence we only need to show that extensions in $\mext^1(\phi(\bu)\cap(\um\setminus\mn(d,\chi)),K_S)^*$ with middle terms not torsion free form a substack whose class is equivalent to $\bl^{d}[\mr^a(1,d-1;-1)(\chi)]$ in $\widehat{K}(Var_{\bc})$ modulo $\widehat{K}_{d^2}$.

Recall that 
\[\um=\mm(d,\chi)\cup\mr^a(1,d-1;-1)(\chi)\cup\mr^a(d-1,1;\chi+2)(\chi)\cup\mr^a(d-1,1;\chi+3)(\chi).\]

Let $\mf\in\mr^a(d-1,1;\chi-s-1)(\chi)$ for $s=-3,-4$.  Since $h^1(\mo_H(s))=-s-1\geq 2$, by (\ref{fex1}) and Lemma \ref{tfcr} (2) (3) the extensions $\eta=[0\ra K_S\ra\widetilde{I}\ra\mf\ra0]$ such that $\widetilde{I}$ contain torsions form a closed subset of codimension $\geq 2$ in $\Ext^1(\mf,K_S)$.  

By Proposition \ref{clor}, $\dim\mr^a(d-1,1)(\chi)=d-1+1+(d-1)^2=d^2-(d-1)$.  By (\ref{phied}) we have 
\begin{eqnarray}\dim\phi^{-1}\big(\mr^a(d-1,1;\chi-s-1)(\chi)\cap\phi(\bu)\big)&\leq& d+1+\dim\mr^a(d-1,1)(\chi)\nonumber\\&=&d+1+d^2-(d-1)=d^2+2.\nonumber\end{eqnarray}
Therefore for $s=-3,-4$
\[\dim\phi^{-1}\left(\big(\mr^a(d-1,1;\chi-s-1)(\chi)\cap\phi(\bu)\big)\setminus\bu\right)\leq d^2,\]
and
\begin{equation}\label{eqtf}[\phi^{-1}\big(\mr^a(d-1,1;\chi-s-1)(\chi)\cap\phi(\bu)\big)\setminus\bu]\equiv0 ~~(\widehat{K}_{d^2}).\end{equation}

For every $\mf\in\mr^a(1,d-1)(\chi)\cap\mm(d,\chi)$, we have two exact sequences as follows.
\[0\ra\mg_1\ra\mf\ra\mo_{H}(s)\ra0,\]
\[0\ra\mo_H(t)\ra\mf\ra\mg_2\ra0,\]
where $s\geq -2$, $-d-1\leq t\leq -3$. 

We have $h^1(\mf)\geq-\chi=d+1$ and $\chi(\mg_2)=\chi-t-1\geq \chi+2=-d+1$.  If $h^1(\mg_2)\leq d+1-2$, then by Lemma \ref{tfcr} (2) all extensions with middle term not torison-free form a closed subset of codimension $\geq 2$ in $\Ext^1(\mf,K_S)$.

If $h^1(\mg_2)\geq d+1-1$, then $h^0(\mg_2)\geq 1\Leftrightarrow \mg_2\in\bigcup_{j\leq 0}\mv_{l,j}(d-1,\chi-t-1)$.

By Proposition \ref{dnkip} (1), we have for $-d-1\leq t\leq -3$
$$\dim\bigcup_{j\leq 0}\mv_{l,j}(d-1,\chi-t-1)\leq (d-1)^2-1+\chi-t-1\leq (d-1)^2-2.$$
Hence all those $\mf$ with $h^1(\mg_2)\geq d+1-1$ form a closed substack of codimension $\geq 2$ in $\mr^a(1,d-1)(\chi)\cap\mm(d,\chi)$.  Therefore by (\ref{phied})
\begin{eqnarray}&&\dim\left( \phi^{-1}\big(\mr^a(1,d-1)(\chi)\cap\mm(d,\chi)\cap\phi(\bu)\big)\setminus\bu\right)\nonumber\\&\leq&\dim \phi^{-1}\big(\mr^a(1,d-1)(\chi)\cap\mm(d,\chi)\cap\phi(\bu)\big)-2\nonumber\\ &\leq& \dim \mr^a(1,d-1)(\chi)+d+1-2=d^2\nonumber\end{eqnarray}

Therefore \begin{equation}\label{eqmm}[\phi^{-1}(\mm(d,\chi)\cap\phi(\bu))\setminus\bu]\equiv 0 ~~(\widehat{K}_{d^2}).\end{equation}

Finally for every $\mf\in\mr^a(1,d-1;-1)(\chi)\cap\phi(\bu)$, we have an exact sequences as follows
\[0\ra\mo_H(-2)\ra\mf\ra\mg_2\ra0.\]
Since $\mf\in\phi(\bu)$, $h^1(\mf)=-\chi=d+1$ and $H^1(\mg_2)\not\cong H^1(\mf)$.  Hence $h^1(\mf)-h^1(\mo_H(-2))\leq h^1(\mg_2)< h^1(\mf)\Rightarrow h^1(\mg_2)=h^1(\mf)-1=d$.  Hence by Lemma \ref{tfcr} (1) (2), extensions with middle terms not torsion free form a codimension one subset of $\Ext^1(\mf,K_S)$ isomorphic to either $\mathbb{A}^d$ or $\mathbb{A}^d\cup\mathbb{A}^1$.

Since $\dim\mr^a(d-1,1)(\chi)=d^2-(d-1)<d^2$, we have \begin{eqnarray}\label{eqtwo}[\phi^{-1}(\mr^a(1,d-1;-1)(\chi)\cap\phi(\bu))\setminus\bu]&\equiv& \bl^d [\mr^a(1,d-1;-1)(\chi)\cap\phi(\bu)]~~(\widehat{K}_{d^2})\nonumber\\ &\equiv&\bl^d [\mr^a(1,d-1;-1)(\chi)]~~(\widehat{K}_{d^2}).
\end{eqnarray}
The last equality is because by Proposition \ref{msbu} we have
\[[\mr^a(1,d-1;-1)(\chi)\cap\phi(\bu)]\equiv [\mr^a(1,d-1;-1)(\chi)]~~(\widehat{K}_{d^2-(d+2)}).\]

By (\ref{eqtf}) (\ref{eqmm}) and (\ref{eqtwo}) and we are done with the proposition.
\end{proof}

Now by (\ref{modhh}) we have
\begin{equation}\label{eq1}[\bu]=[\bh^0(\mh^{n,0,0}_L(d))^*]\equiv\bl^{2d}\cdot\left([\mh^n]-[\mh_L^{n,0,1}]\right)~~(\widehat{K}_{d^2}).\end{equation}
Moreover by Proposition \ref{stlo} we have
\begin{eqnarray}\label{eq2}[\bu]&\equiv&\bl^{d+1}[\mm(d,\chi)]+(\bl^{d+1}-\bl^d)[\mr^a(1,d-1;-1)(\chi)]\nonumber\\
&&+\bl^{d+1}[\mr^a(d-1,1;\chi+2)(\chi)]+\bl^{d+1}[\mr^a(d-1,1;\chi+3)(\chi)]~~(\widehat{K}_{d^2}).\end{eqnarray}

Combine (\ref{eq1}) and (\ref{eq2}), and we have
\begin{eqnarray}\label{eq3}\bl^{d}\left([\mh^n]-[\mh_L^{n,0,1}]\right)&\equiv&\bl[\mm(d,\chi)]+(\bl-1)[\mr^a(1,d-1;-1)(\chi)]\nonumber\\
&&+\bl[\mr^a(d-1,1;\chi+2)(\chi)]+\bl[\mr^a(d-1,1;\chi+3)(\chi)]~~(\widehat{K}_{d^2-d}).\end{eqnarray}

\begin{prop}\label{mshn01}In $\widehat{K}(Var_{\bc})$, $\bl^d[\mh_L^{n,0,1}]\equiv[\mr^a(d-1,1;\chi+4)(\chi)]~~(\widehat{K}_{d^2-d-1}).$
\end{prop}
\begin{proof}By (\ref{deh}) we have
\begin{equation}\label{deh01}[\bh^0(\mh_L^{n,0,1}(d))^*]=(\bl^{2d+1}-1)[\mh_L^{n,0,1}].\end{equation}
Since $\dim\mh^{n,0,1}\leq 2n-1-d=d^2-2d+1$ by Proposition \ref{dnkip} (4), we have
\[[\bh^0(\mh_L^{n,0,1}(d))^*]\equiv\bl^{2d+1}[\mh_L^{n,0,1}]~~(\widehat{K}_{d^2}).
\]

Define $\bu_1:=\Psi^{-1}\left(\bh^0(\mh_L^{n,0,1}(d))^*\right)\subset\mext^1(\mm^a_{\bullet}(d,\chi),K_S)$.  It is enough to show $[\bu_1]\equiv\bl^{d+1}[\mr^a(d-1,1;\chi+4)(\chi)]~~(\widehat{K}_{d^2}).
$

For every $\mf\in\phi(\bu_1)$, there is an exact sequence
\[0\ra K_S\ra I_n(d-3)\ra\mf\ra0,\]
with $h^0(I_n(d-3))=0,~h^1(I_n(d))=1$.  Hence for any $\mf\in\phi(\bu_1)$ we have $h^0(\mf)=0, h^1(\mf(3))=1$.  Therefore $\dim\Ext^1(\mf,K_S)=h^1(\mf)=-\chi=d+1,\forall \mf\in\phi(\bu_1)$ and we have
\begin{equation}\label{phied1}[\phi^{-1}(\phi(\bu_1))]=(\bl^{d+1}-1)[\phi(\bu_1)].\end{equation}

Recall that by (\ref{codtn}) we have $\dim \left(\mm^a_{\bullet}(d,\chi)\setminus\wmn^a(d,\chi)\right)\leq d^2-2(d-2)$.  Therefore 
\begin{eqnarray}\dim\left(\bu_1\setminus\phi^{-1}\left(\phi(\bu_1)\cap\wmn^a(d,\chi)\right)\right)&\leq&\dim\phi^{-1}\left(\phi(\bu)\setminus\wmn^a(d,\chi)\right)\nonumber\\
&=&d+1+\dim \left(\phi(\bu)\setminus\wmn^a(d,\chi)\right)\nonumber\\
&\leq&d+1+\dim \left(\mm^a_{\bullet}(d,\chi)\setminus\wmn^a(d,\chi)\right) \nonumber\\&\leq&d^2-2(d-2)+d+1=d^2-d+5\leq d^2.\nonumber\end{eqnarray}
Hence we have
\begin{equation}\label{hneq1}[\bu_1]\equiv[\phi^{-1}\left(\phi(\bu_1)\cap\wmn^a(d,\chi)\right)\cap\bu_1]~~(\widehat{K}_{d^2}).\end{equation}

Since by definition $\phi(\bu)\cap\phi(\bu_1)=\emptyset$, by Proposition \ref{msbu}, we have
\begin{eqnarray}[\phi(\bu_1)\cap\wmn^a(d,\chi)]&\equiv&[\phi(\bu_1)\cap\left(\wmn^a(d,\chi)\setminus\um\right)]~~(\widehat{K}_{d^2-(d+1)})\nonumber\\ &\equiv&[\phi(\bu_1)\cap\bigcup_{s\leq -5}\mr^a(d-1,1;\chi-s-1)(\chi)]~~(\widehat{K}_{d^2-(d+1)})\nonumber.\end{eqnarray}
Because for any $\mf\in\phi(\bu_1)$, $h^0(\mf)=0, h^1(\mf(3))=1$, and for any $\mf\in \mr^a(d-1,1;\chi-s-1)(\chi)$, $\mo_H(s)$ is a quotient of $\mf$,  we have 
\[\phi(\bu_1)\cap\mr^a(d-1,1;\chi-s-1)(\chi)=\emptyset, ~\forall~s\leq -6.\]
Hence
\begin{equation}\label{hneq2}[\phi^{-1}\left(\phi(\bu_1)\cap\wmn^a(d,\chi)\right)]\equiv[\phi^{-1}\big(\phi(\bu_1)\cap\mr^a(d-1,1;\chi+4)(\chi)\big)]~~(\widehat{K}_{d^2}).\end{equation}

On the other hand it is easy to see that 
$$\mr^a(d-1,1,\chi+4)\setminus\phi(\bu_1)\subset \left(\ms_1(-5)\cup\widetilde{\ms_2}(-5)\cup\ms_3(-5)\right),$$
where $\widetilde{\ms_2}(-5)\subset \mr^a(d-1,1,\chi+4)$ parametrizes all sheaves $\mf$ with $h^1(\mf(-K_S))>1$.
Every sheaf $\mf\in\mr^a(d-1,1;\chi+4)(\chi)$ lies in the following sequence
\[0\ra\mg_1\ra\mf\ra\mo_H(-5)\ra0.\]
If $h^1(\mf(-K_S))>1$, then $h^1(\mg_1(-K_S))\geq1\Leftrightarrow \mg_1\in\bigcup_{i\geq 1}\mn_{k,i}(d-1,\chi+4)$.  Therefore by the analogous argument to Lemma \ref{stlr} (2), $\dim\widetilde{\ms_2}(-5)\leq d^2-3d$.
Hence by Lemma \ref{stlr} (1) (3) we have
\begin{equation}\label{hneq3}[\phi(\bu_1)\cap\mr^a(d-1,1;\chi+4)(\chi)]\equiv[\mr^a(d-1,1;\chi+4)(\chi)]~~(\widehat{K}_{d^2-(d+1)}).\end{equation}

Since $h^1(\mo_H(-5))=4$, by Lemma \ref{tfcr} (3) for every $\mf\in\mr^a(d-1,1;\chi+4)(\chi)$, the extensions $\eta=[0\ra K_S\ra\widetilde{I}\ra\mf\ra0]$ such that $\widetilde{I}$ contain torsions form a closed subset of codimension $\geq \min\{4,d+1-4\}\geq 2$ in $\Ext^1(\mf,K_S)$.  Hence by (\ref{hneq1})
\begin{eqnarray}&&\dim \left(\phi^{-1}\big(\mr^a(d-1,1;\chi+r)(\chi)\cap\phi(\bu_1)\big)\setminus\bu_1\right)\nonumber\\ &\leq& d+1-2+\dim \mr^a(d-1,1;\chi+r)(\chi)=d^2\nonumber\end{eqnarray}
Therefore 
\begin{equation}\label{hneq4}[\phi^{-1}\big(\mr^a(d-1,1;\chi+r)(\chi)\cap\phi(\bu_1)\big)]\equiv [\phi^{-1}\big(\mr^a(d-1,1;\chi+r)(\chi)\cap\phi(\bu_1)\big)\cap\bu_1] ~~(\widehat{K}_{d^2}).\end{equation}

Put (\ref{hneq1}) (\ref{hneq2}) (\ref{hneq3}) and (\ref{hneq4}) together and we get that
\begin{eqnarray}[\bu_1]&\equiv& [\phi^{-1}\left(\phi(\bu_1)\cap\wmn^a(d,\chi)\right)\cap\bu_1]~~(\widehat{K}_{d^2})\nonumber\\
&\equiv& [\phi^{-1}\big(\phi(\bu_1)\cap\mr^a(d-1,1;\chi+4)(\chi)\big)\cap\bu_1]~~(\widehat{K}_{d^2})\nonumber\\
&\equiv& [\phi^{-1}\big(\mr^a(d-1,1;\chi+r)(\chi)\cap\phi(\bu_1)\big)]~~(\widehat{K}_{d^2})\nonumber\\
&=&(\bl^{d+1}-1)[\mr^a(d-1,1;\chi+r)(\chi)\cap\phi(\bu_1)]\nonumber\\
&\equiv&\bl^{d+1}[\mr^a(d-1,1;\chi+r)(\chi)\cap\phi(\bu_1)]~~(\widehat{K}_{d^2})\nonumber\\
&\equiv&\bl^{d+1}[\mr^a(d-1,1;\chi+4)(\chi)]~~(\widehat{K}_{d^2}).
\end{eqnarray}
The proposition is proved.
\end{proof}
\begin{thm}\label{main}In $\widehat{K}(Var_{\bc})$ we have
\[(\bl^{d-1}[S^{[n]}]-[M(d,\chi)])\equiv3[\p^{2d-4}][\p^2][S^{[n_1]}]~~(\widehat{K}_{d^2-d}),\]
where $n_1=\frac{(d-1)(d-2)}2+1=n-(d-1)$.
In particular,\[b_{2k}(M(d,\chi))=\left\{\begin{array}{ll}b_{2k}(S^{[n]}), &k\leq d-2\\ b_{2k}(S^{[n]})-3, &k=d-1\\b_{2k}(S^{[n]})-12,&k= d\end{array}\right..\]
\end{thm}
\begin{proof}By (\ref{eq3}) and Proposition \ref{mshn01} we have
\begin{eqnarray}\label{main1}\bl^{d}[\mh^n]&\equiv&\bl[\mm(d,\chi)]+(\bl-1)[\mr^a(1,d-1;-1)(\chi)]\nonumber\\
&&+\bl[\mr^a(d-1,1;\chi+2)(\chi)]+\bl[\mr^a(d-1,1;\chi+3)(\chi)]\nonumber\\
&&+[\mr^a(d-1,1;\chi+4)(\chi)]~~(\widehat{K}_{d^2-d}).\end{eqnarray}

Since $\dim\mr^a(1,d-1)(\chi)=d^2-(d-1)$ by Proposition \ref{clor}, (\ref{main1}) implies that $$\bl^{d}[\mh^n]\equiv\bl\mm(d,\chi)~~(\widehat{K}_{d^2-d+2}).$$Hence
\begin{equation}\label{main2}\bl^{d-1}[\mh^n]\equiv\mm(d,\chi)\equiv\mn(d,\chi)~~(\widehat{K}_{d^2-d+1}).
\end{equation}
Hence by Lemma \ref{fdms} we reproved Theorem \ref{genY} for $d\geq 5$ and $\chi=-d-1$.  By \cite[Theorem 6.11 and Corollary 7.1]{Yuan4} for any $d'\geq 4$, $n'=\frac{d'(d'-1)}2+1$ and $\chi_1,\chi_2$, we have
\begin{equation}\label{main3}\bl^{d'-1}[\mh^{n'}]\equiv\mm(d',\chi_1)\equiv\mn(d',\chi_1)\equiv\mm(d',\chi_2)\equiv\mn(d',\chi_2)~~(\widehat{K}_{d'^2-3}).
\end{equation}

By Proposition \ref{clor} we have
\[[\mr^a(L_1,L_2;\chi_1)(\chi)]=\bl^{L_1.L_2}\frac{[N(L_1,\chi_1)]}{\bl-1}\frac{[N(L_2,\chi-\chi_1)]}{\bl-1}.\]
Notice that for any $\chi'$, $N(1,\chi')\cong M(1,\chi')\cong\p^2=S$.  By (\ref{main3}) we have $[N(d-1,\chi')]\equiv \bl^{d-2}[S^{[n_1]}]~~(\widehat{K}_{d^2-2d-1})$ with $n_1=\frac{(d-1)(d-2)}2+1$.  Therefore we have for $i=2,3,4$
\[[\mr^a(1,d-1;-1)(\chi)]\equiv[\mr^a(d-1,1;\chi+i)(\chi)]\equiv \bl^{2d-3}\frac{[\p^2]\cdot [S^{[n]}]}{(\bl-1)^2}~~(\widehat{K}_{d^2-d-2}).
\]
Hence by (\ref{main1}) we have
\begin{eqnarray}\label{mt4r}&&\bl^{d}[\mh^n]-\bl[\mm(d,\chi)]\nonumber\\
&\equiv&
(\bl-1)[\mr^a(1,d-1;-1)(\chi)]+\bl[\mr^a(d-1,1;\chi+2)(\chi)]\nonumber\\
&&+\bl[\mr^a(d-1,1;\chi+3)(\chi)]+[\mr^a(d-1,1;\chi+4)(\chi)]
\nonumber\\
&\equiv&3\bl\cdot[\mr^a(1,d-1;-1)(\chi)]\equiv 3\bl^{2d-2}\frac{[\p^2]\cdot [S^{[n]}]}{(\bl-1)^2}~~(\widehat{K}_{d^2-d})
\end{eqnarray}
Multiply both side of (\ref{mt4r}) by $\frac{\bl-1}{\bl}$ and we get 
\begin{eqnarray}\label{main4}&&(\bl^{d-1}[S^{[n]}]-[M(d,\chi)])\nonumber\\
&\equiv&3\bl^{2d-3}\frac{[\p^2][S^{[n_1]}]}{\bl-1}~~(\widehat{K}_{d^2-d}).\end{eqnarray}
Since $\dim\frac{\p^2\times S^{[n_1]}}{\bl-1}=d^2-3d+5\leq d^2-d$ for $d\geq 5$, (\ref{main4}) implies
\begin{eqnarray}\label{main5}&&(\bl^{d-1}[S^{[n]}]-[M(d,\chi)])\\
&\equiv&3\bl^{2d-3}\frac{[\p^2][S^{[n_1]}]}{\bl-1}-3\frac{\p^2\times S^{[n_1]}}{\bl-1}~~(\widehat{K}_{d^2-d})\nonumber\\
&=&3(\bl^{2d-3}-1)\frac{[\p^2][S^{[n_1]}]}{\bl-1}=3[\p^{2d-4}][\p^2][S^{[n_1]}].\nonumber\end{eqnarray}

Recall that $P_v(X;z):=\sum b^v_{i}(X)z^i=H_c(X,z,z,-1)$ is the virtual Poincar\'e polynomial of $X$. The function $[X]\mapsto P_v(X;z)$ is a motivic measure on $K(Var_{\bc})$ taking values in $\bz[z]$, which can be extended to $\widehat{K}(Var_{\bc})$.  For $[\mm]\in\widehat{K}(Var_{\bc})$, $P_v(\mm;z)\in\bz[z][z^{-2},(z^{2i}-1)^{-1}:i\geq 1]$ and $P_v(\bl;z)=z^2$.

By Lemma \ref{fdms}, the rational function (a polynomial in fact) 
\begin{equation}\label{vppe}P_v(M(d,\chi);z)-z^{2d-2}P_v(S^{[n]})+3P_v(\p^{2d-4};z)P_v(\p^2;z)P_v(S^{[n_1]};z)\end{equation} is  of degree $\leq 2(d^2-d)$.  Since all varieties in (\ref{vppe}) are smooth and complete, their virtual Poincar\'e polynomials coincide with the ordinary ones.  By G\"ottsche's formula (\ref{genG}) and Poincar\'e duality one can compute that $b_{2}(S^{[n_1]})=b_{4n_1-2}(S^{[n_1]})=2$.  Therefore by a direct computation we have
\[b_{2k}(M(d,\chi))=\left\{\begin{array}{ll}b_{2k-2d+2}(S^{[n]}),&k\geq (d^2+1-(d-2))\\
b_{2k-2d+2}(S^{[n]})-3, &k=d^2+1-(d-1)\\ b_{2k-2d+2}(S^{[n]})-12, &k=d^2+1-d\end{array}\right.,\]
and for $2(d^2-d)+1\leq i\leq 2(d^2+1)$ and $i$ odd, $b_i(M(d,\chi))=0$ (we already know that all the odd Betti numbers of $M(d,\chi)$ are zero due to Markman's result).

Since $\dim M(d,\chi)=d^2+1$, by Poincar\'e duality we have
\[b_{2k}(M(d,\chi))=\left\{\begin{array}{ll}b_{2k}(S^{[n]}),&k\leq d-2\\ b_{2k}(S^{[n]})-3, &k=d-1\\b_{2k}(S^{[n]})-12, &k= d\end{array}\right..\]
The theorem is proved.
\end{proof}
The following corollary proves Conjecture 3.3 in \cite{PS}.
\begin{coro}\label{mainco}For $d\geq 5$ and $\chi$ coprime to $d$, the $3d-7$ generators 
$$c_0(2),c_2(0),c_k(0),c_{k-1}(1),c_{k-2}(2), k\in {3,\cdots,d-1}$$ of $A^*(M(d,\chi))\cong H^*(M(d,\chi),\mathbb{Z})$ given in \cite{PS} have no relation in $A^i(M(d,\chi)),i\leq d-1$ and have 3 linearly independent relations in $A^d(M(d,\chi))$. 
\end{coro}
\begin{rem}We may do more dimension estimate to get more Betti numbers such as $b_{2(d+1)}, b_{2(d+2)}$.  However the computation could become extremely complicated and it is not hopeful to get all Betti numbers in this way.  Nevertheless, our dimension estimate could be useful somewhere else.  
\end{rem}

Lemma \ref{fdms} should be a standard fact to experts (see e.g. \cite[Lemma 6.1.1]{HRV}).  We give a short proof here as well.  

\begin{lemma}\label{fdms}Let $P_v(\mm;z)$ be the virtual Poincar\'e polynomial of $\mm\in\widehat{K}(Var_{\bc})$.  For a rational function $\frac{f(z)}{g(z)}\in\bz[z][z^{-2},(z^{2i}-1)^{-1}:i\geq 1]$, define \emph{deg} $\frac{f(z)}{g(z)}:=$\emph{ deg }$f(z)-$\emph{ deg }$g(z)$.  Then for any $\mm\in\widehat{K}_m$, \emph{deg} $P_v(\mm;z)\leq 2m$.
\end{lemma}
\begin{proof}It is enough to show that deg $P_v(X;z)\leq 2m$ for any variety $X$ with $\dim X\leq m$.  We do the induction on $m$.  If $m=0$, then it is trivial.

By induction assumption, we only need to show deg $P_v(X;z)\leq 2m$ for $X$ quasiaffine and smooth.  By resolution, we can find a smooth projective scheme $Y$ as a compactification of $X$.  Therefore by induction assumption deg $P_v(Y\setminus X,z)\leq 2(m-1)$.  Since $Y$ is projective smooth, deg $P_v(Y)=2\dim Y=2\dim X\leq 2m$.  Hence
$P_v(X;z)=P_v(Y;z)-P_v(Y\setminus X;z)$ and deg $P_v(X;z)\leq 2m$.
The lemma is proved.    
\end{proof}


Yao Yuan\\
Beijing National Center for Applied Mathematics,\\
Academy for Multidisciplinary Studies, \\
Capital Normal University, 100048, Beijing, China\\
E-mail: 6891@cnu.edu.cn.
\end{document}